\documentclass[11pt]{amsart}

\usepackage{amsfonts}
\usepackage{mathrsfs}
\usepackage{amscd}
\usepackage{amsmath}
\usepackage{amssymb}
\usepackage{latexsym}
\usepackage{lscape}
\usepackage{xypic}
\usepackage{comment}
\usepackage{amscd}
\usepackage{wasysym}  
\usepackage{stmaryrd}  
\usepackage{tikz} 
\usetikzlibrary{matrix,arrows}
\usepackage{appendix}
\usepackage{geometry}
\usepackage[nice]{nicefrac}
\usepackage{dsfont}
\usepackage{thmtools}

\usepackage[]{overpic}

\usepackage[pagebackref,hyperindex,linktocpage=true]{hyperref}
\hypersetup{
    colorlinks,
    linkcolor={red!50!black},
    citecolor={blue!50!black},
    urlcolor={blue!80!black}
}

\usepackage{color}
\newcommand{\cred}{\color{red}}

\newcommand{\cb}{\color{blue}}
\newcommand{\cm}{\color{magenta}}
\newcommand{\cg}{\color{green}}
\newcommand{\cp}{\color{pink}}

\newtheorem{thm}{Theorem}[section]

\newtheorem{prop}[thm]{Proposition}
\newtheorem{lem}[thm]{Lemma}

\newtheorem{lem-def}[thm]{Lemma-Definition}
\newtheorem{cor}[thm]{Corollary}

\newtheorem{thmx}{Theorem}

\theoremstyle{definition}

\newtheorem{ex}[thm]{Example}

\newtheorem{rmk}[thm]{Remark}
\newtheorem{dfn}[thm]{Definition}

\numberwithin{equation}{section}

\newcommand{\nc}{\newcommand}

\nc{\on}{\operatorname}
\nc{\fraka}{{\mathfrak a}} \nc{\bba}{{\mathbf a}}
\nc{\frakb}{{\mathfrak b}}
\nc{\frakc}{{\mathfrak c}}
\nc{\frakd}{{\mathfrak d}}
\nc{\frake}{{\mathfrak e}}
\nc{\frakf}{{\mathfrak f}}
\nc{\frakg}{{\mathfrak g}}
\nc{\frakh}{{\mathfrak h}}
\nc{\fraki}{{\mathfrak i}}
\nc{\frakj}{{\mathfrak j}}
\nc{\frakk}{{\mathfrak k}}
\nc{\frakl}{{\mathfrak l}}
\nc{\frakm}{{\mathfrak m}}
\nc{\frakn}{{\mathfrak n}}
\nc{\frako}{{\mathfrak o}}
\nc{\frakp}{{\mathfrak p}}
\nc{\frakq}{{\mathfrak q}}
\nc{\frakr}{{\mathfrak r}}
\nc{\fraks}{{\mathfrak s}}
\nc{\frakt}{{\mathfrak t}}
\nc{\fraku}{{\mathfrak u}}
\nc{\frakv}{{\mathfrak v}}
\nc{\frakw}{{\mathfrak w}}
\nc{\frakx}{{\mathfrak x}}
\nc{\fraky}{{\mathfrak y}}
\nc{\frakz}{{\mathfrak z}}
\nc{\frakA}{{\mathfrak A}}
\nc{\frakB}{{\mathfrak B}}
\nc{\frakC}{{\mathfrak C}}
\nc{\frakD}{{\mathfrak D}}
\nc{\frakE}{{\mathfrak E}}
\nc{\frakF}{{\mathfrak F}}
\nc{\frakG}{{\mathfrak G}}
\nc{\frakH}{{\mathfrak H}}
\nc{\frakI}{{\mathfrak I}}
\nc{\frakJ}{{\mathfrak J}}
\nc{\frakK}{{\mathfrak K}}
\nc{\frakL}{{\mathfrak L}}
\nc{\frakM}{{\mathfrak M}}
\nc{\frakN}{{\mathfrak N}}
\nc{\frakO}{{\mathfrak O}}
\nc{\frakP}{{\mathfrak P}}
\nc{\frakQ}{{\mathfrak Q}}
\nc{\frakR}{{\mathfrak R}}
\nc{\frakS}{{\mathfrak S}}
\nc{\frakT}{{\mathfrak T}}
\nc{\frakU}{{\mathfrak U}}
\nc{\frakV}{{\mathfrak V}}
\nc{\frakW}{{\mathfrak W}}
\nc{\frakX}{{\mathfrak X}}
\nc{\frakY}{{\mathfrak Y}}
\nc{\frakZ}{{\mathfrak Z}}
\nc{\bbA}{{\mathbb A}}
\nc{\bbB}{{\mathbb B}}
\nc{\bbC}{{\mathbb C}}
\nc{\bbD}{{\mathbb D}}
\nc{\bbE}{{\mathbb E}}
\nc{\bbF}{{\mathbb F}} \nc{\bbf}{{\mathbf f}}
\nc{\bbG}{{\mathbb G}}
\nc{\bbH}{{\mathbb H}}
\nc{\bbI}{{\mathbb I}}
\nc{\bbJ}{{\mathbb J}}
\nc{\bbK}{{\mathbb K}}
\nc{\bbL}{{\mathbb L}}
\nc{\bbM}{{\mathbb M}}
\nc{\bbN}{{\mathbb N}}
\nc{\bbO}{{\mathbb O}}
\nc{\bbP}{{\mathbb P}}
\nc{\bbQ}{{\mathbb Q}}
\nc{\bbR}{{\mathbb R}}
\nc{\bbS}{{\mathbb S}}
\nc{\bbT}{{\mathbb T}}
\nc{\bbU}{{\mathbb U}}
\nc{\bbV}{{\mathbb V}}
\nc{\bbW}{{\mathbb W}}
\nc{\bbX}{{\mathbb X}}
\nc{\bbY}{{\mathbb Y}}
\nc{\bbZ}{{\mathbb Z}}
\nc{\calA}{{\mathcal A}}
\nc{\calB}{{\mathcal B}}
\nc{\calC}{{\mathcal C}}
\nc{\calD}{{\mathcal D}}
\nc{\calE}{{\mathcal E}}
\nc{\calF}{{\mathcal F}}
\nc{\calG}{{\mathcal G}}
\nc{\calH}{{\mathcal H}}
\nc{\calI}{{\mathcal I}}
\nc{\calJ}{{\mathcal J}}
\nc{\calK}{{\mathcal K}}
\nc{\calL}{{\mathcal L}}
\nc{\calM}{{\mathcal M}}
\nc{\calN}{{\mathcal N}}
\nc{\calO}{{\mathcal O}}
\nc{\calP}{{\mathcal P}}
\nc{\calQ}{{\mathcal Q}}
\nc{\calR}{{\mathcal R}}
\nc{\calS}{{\mathcal S}}
\nc{\calT}{{\mathcal T}}
\nc{\calU}{{\mathcal U}}
\nc{\calV}{{\mathcal V}}
\nc{\calW}{{\mathcal W}}
\nc{\calX}{{\mathcal X}}
\nc{\calY}{{\mathcal Y}}
\nc{\calZ}{{\mathcal Z}}

\nc{\scrA}{{\mathscr A}}
\nc{\scrB}{{\mathscr B}}
\nc{\scrR}{{\mathscr R}}

\nc{\bnu}{{\bar{ \nu}}}

\nc{\olO}{\bar{\calO}}

\nc{\al}{{\alpha}} 
\nc{\be}{{\beta}}
\nc{\ga}{{\gamma}} \nc{\Ga}{{\Gamma}}
 \nc{\hGa}{\hat{\Gamma}}
\nc{\ve}{{\varepsilon}} 
\nc{\la}{{\lambda}} \nc{\La}{{\Lambda}}
\nc{\om}{\omega} \nc{\Om}{\Omega} 
\nc{\sig}{{\sigma}} \nc{\Sig}{{\Sigma}}

\nc{\tnb}{\psi_{\rm tame}}
\nc{\oM}{\overline{{M}}}
\nc{\op}{{\on{op}}}
\nc{\ad}{{\on{ad}}}
\nc{\alg}{{\on{alg}}}
\nc{\Ad}{{\on{Ad}}}
\nc{\Adm}{{\on{Adm}}} \nc{\aff}{{\on{af}}}
\nc{\Aut}{{\on{Aut}}}
\nc{\Bun}{{\on{Bun}}}
\nc{\cha}{{\on{char}}}
\nc{\der}{{\on{der}}}
\nc{\Der}{{\on{Der}}}
\nc{\diag}{{\on{diag}}}
\nc{\End}{{\on{End}}}
\nc{\Fl}{{\calF\!\ell}}
\nc{\Tr}{{\on{Transp}}}
\nc{\TR}{{\calT\!\calR}}
\nc{\Gal}{{\on{Gal}}}
\nc{\Gr}{{\on{Gr}}}
\nc{\rH}{{\on{H}}}
\nc{\Hom}{{\on{Hom}}}
\nc{\IC}{{\on{IC}}}
\nc{\id}{{\on{id}}}
\nc{\Id}{{\on{Id}}}
\nc{\ind}{{\on{ind}}}
\nc{\Ind}{{\on{Ind}}}
\nc{\Lie}{{\on{Lie}}}
\nc{\Pic}{{\on{Pic}}}
\nc{\pr}{{\on{pr}}}
\nc{\Res}{{\on{Res}}}
\nc{\res}{{\on{res}}} \nc{\Sat}{{\on{Sat}}}
\nc{\s}{{\on{sc}}}
\nc{\drv}{{\on{der}}}
\nc{\sgn}{{\on{sgn}}}
\nc{\Spec}{{\on{Spec}}}\nc{\Spf}{\on{Spf}} 
\nc{\Sph}{\on{Sph}}
\nc{\St}{{\on{St}}}
\nc{\tr}{{\on{tr}}}
\nc{\Mod}{{\mathrm{-Mod}}}
\nc{\Hilb}{{\on{Hilb}}} 
\nc{\Ext}{{\on{Ext}}} 
\nc{\vs}{{\on{Vec}}}
\nc{\ev}{{\on{ev}}}
\nc{\nO}{{\breve{\calO}}}
\nc{\tS}{{\tilde{S}}}
\nc{\spe}{{\on{sp}}}
\nc{\loc}{{\on{loc}}}

\nc{\nscrR}{{\mathscr{R}^{\on{nr}}}}

\nc{\GL}{{\on{GL}}}
\nc{\U}{{\on{U}}}
\nc{\Gl}{\on{Gl}} 
\nc{\GSp}{{\on{GSp}}}
\nc{\gl}{{\frakg\frakl}}
\nc{\SL}{{\on{SL}}} 
\nc{\SU}{{\on{SU}}} 
\nc{\SO}{{\on{SO}}}
\nc{\PGL}{{\on{PGL}}}

\nc{\Conv}{{\on{Conv}}}
\nc{\Rep}{{\on{Rep}}}
\nc{\Dom}{{\on{Dom}}}
\nc{\red}{{\on{red}}}
\nc{\act}{{\on{act}}}
\nc{\nr}{{\on{nr}}}
\nc{\ctf}{{\on{ctf}}}

\nc{\str}{{\on{-}}} 
\nc{\os}{{\bar{s}}}
\nc{\oeta}{{\bar{\eta}}}

\nc{\hookto}{\hookrightarrow}
\nc{\longto}{\longrightarrow}
\nc{\leftto}{\leftarrow}
\nc{\onto}{\twoheadrightarrow}
\nc{\lonto}{\twoheadleftarrow}

\nc{\uG}{{\underline{G}}}
\nc{\uA}{{\underline{A}}}
\nc{\uS}{{\underline{S}}}
\nc{\uT}{{\underline{T}}}
\nc{\uM}{{\underline{M}}}
\nc{\uP}{{\underline{P}}}
\nc{\uB}{{\underline{B}}}
\nc{\uN}{{\underline{N}}}

\nc{\ucG}{{\underline{\calG}}}
\nc{\ucA}{{\underline{\calA}}}
\nc{\ucS}{{\underline{\calS}}}
\nc{\ucT}{{\underline{\calT}}}
\nc{\ucM}{{\underline{\calM}}}
\nc{\ucP}{{\underline{\calP}}}
\nc{\ucN}{{\underline{\calN}}}

\nc{\bF}{{\breve{F}}}

\nc{\oFl}{{\overline{\Fl}}} 
\nc{\bU}{{\overline{U}}}
\nc{\tGr}{{\tilde{\Gr}}}
\nc{\cGr}{\calG\! r}
\nc{\oGr}{\overline{\on{Gr}}} 
\nc{\ocGr}{\overline{\calG\! r}}
\nc{\co}{{\colon}}
\nc{\sch}[1]{(Sch/{#1})}
\nc{\HypLoc}[1]{HypLoc({#1})}

\nc{\ohtimes}{\stackrel{!}{\otimes}}
\nc{\boxtilde}{\widetilde{\boxtimes}}
\nc{\vstar}{{\varhexstar}}

\nc{\Div}{\on{Div}}

\nc{\bslash}{\backslash}
\nc{\algQl}{{\bar{\bbQ}_\ell}}
\nc{\sF}{{\bar{F}}}
\nc{\nF}{{\breve{F}}}
\nc{\nW}{{W^{\on{nr}}}}
\nc{\sk}{{\bar{k}}}
\nc{\cont}{\on{c}}
\nc{\Supp}{\on{Supp}}
\nc{\blt}{\bullet}  
\nc{\dom}{\on{dom}}
\nc{\scon}{{\on{sc}}} 
\nc{\Affine}{\on{Aff}} 
\nc{\nscrA}{\mathscr{A}^{\on{nr}}} 
\nc{\nfraka}{{\bbf^{\on{nr}}}}
\nc{\ran}{{\rangle}}
\nc{\lan}{{\langle}}
\nc{\bk}{{\bar{k}}}
\nc{\tF}{{\tilde{F}}}
\nc{\sS}{{\bar{S}}}
\nc{\LG}{{^\text{L}\hspace{-0.04cm}G}}
\nc{\LL}{{^\text{L}\hspace{-0.07cm}L}}

\nc{\pot}[1]{ [\hspace{-0,5mm}[ {#1} ]\hspace{-0,5mm}] }
\nc{\rpot}[1]{ (\hspace{-0,7mm}( {#1} )\hspace{-0,7mm}) }

\nc{\defined}{\hspace{0.1cm}\stackrel{\text{\tiny \rm def}}{=}\hspace{0.1cm}}

\begin{document}

\title[Alcove walk models for Mirkovi\'c-Vilonen intersections and Levi branching]{Alcove walk models for parabolic Mirkovi\'c-Vilonen intersections and branching to Levi subgroups}
\author[T.\,J.\,Haines]{by Thomas J. Haines}

\address{Department of Mathematics, University of Maryland, College Park, MD 20742-4015, DC, USA}
\email{tjh@umd.edu}



\thanks{\iffalse {\it 2010 Mathematics Subject Classification} 14G35(???????).\fi Research of T.H.~partially supported by NSF DMS-2200873. Disclaimer: Any opinions, findings, and conclusions or recommendations expressed in this material are those of the author and do not necessarily reflect the views of the National Science Foundation.}

\maketitle

\begin{abstract}
This article proves an alcove walk description of intersections of Schubert cells and partially semi-infinite orbits (depending on a choice of parabolic subgroup) in the affine Grassmannian of a split connected reductive group (we call these intersections {\em parabolic Mirkovi\'c-Vilonen intersections}).  We deduce from this a parameterization of the irreducible components of the maximal possible dimension in these intersections by the alcove walks of maximal possible dimension. As a consequence we present a new combinatorial description of branching to Levi subgroups of irreducible highest weight representations, and in particular, we give a new algorithm for computing the characters of such representations. 
\smallskip


\end{abstract}

\setcounter{tocdepth}{1}
\tableofcontents
\setcounter{section}{0}

\thispagestyle{empty}

\section{Introduction}

\subsection{Statement of Main Results}
Let $G$ be a connected reductive algebraic group, defined and split over any field $k$, and fix a $k$-split Borel pair $B \supset T$ and a Levi decomposition $B = TU$.  For any cocharacters $\mu, \lambda \in X_*(T)$, we consider the Mirkovi\'c-Vilonen (reduced) intersection $$LU\,t^\lambda L^+G/L^+G ~\cap~ L^+G\,t^\mu L^+G/L^+G$$ in the $t$-power series affine Grassmannian ${\rm Gr}_G$.  Such varieties play an essential role in the geometric Satake equivalence \cite{MV07, NP01, Ri14}.  More generally, for any parabolic subgroup $P \supset B$ with Levi factor $M \supset T$ and unipotent radical $N \subset U$, we may consider the {\em parabolic Mirkovi\'c-Vilonen intersections} $$L^+M LN\,t^\lambda L^+G/L^+G ~\cap~ L^+G\,t^\mu L^+G/L^+G$$ in ${\rm Gr}_G$.  These intersections, their close relatives in the (sometimes Witt vector) affine flag varieties, and the methods used to study them, have been employed in the study of affine Deligne-Lusztig varieties and special fibers of certain Shimura varieties (e.g.\,\cite{GHKR06, GHKR10, XZ17+, Ni22}), and in geometrizations of and combinatorial and crystal properties of constant term homomorphisms (e.g.,\,\cite{BD, BG01, CP24, HKM}).

The current paper uses retractions in the Bruhat-Tits building and alcove walks  in the base apartment as the primary tools, relying in key places on previous works \cite{HKM, Hai25}. This approach has several precedents.  Retractions themselves are fundamental aspects of the theory of buildings and appeared in foundational works of Bruhat-Tits \cite{BT72}, and alcove walks have long been known to provide a helpful language to visualize words in Coxeter groups (see e.g.\,\cite{HN02}). Ram's definition of the alcove walk algebra in \cite{Ram06} led to advances in the study of affine Hecke algebras (e.g.\,\cite{Goe07}),  yielding integral presentations exploited in the mod $\ell$ and mod $p$ local Langlands programs by Vigneras \cite{Vig16} and others. Concerning affine Deligne-Lusztig varieties, the alcove walk technique first appeared in \cite{GHKR06, GHKR10}, and was further developed in more recent works, e.g.\,\cite{MST19, MST23, MNST24}.

Gaussent and Littelmann introduced in \cite{GL05} a gallery-theoretic analogue of the Littelmann Path model \cite{Li95}, and alcove walks are special cases of the galleries they consider. The Gaussent-Littelmann model in particular has been used in recent works on the (integral) motivic geometric Satake equivalence (\cite{CSvdH22}, \cite{vdH24}). Many further applications of alcove walks and galleries to Hecke algebras, Hall-Littlewood polynomials, and combinatorial representation theory have appeared (for a sample, see \cite{Ram06, PRS09}).

The first main theorem of this paper is the following description of the pieces in a cellular paving of parabolic Mirkovi\'c-Vilonen intersections, indexed by an explicit set $\mathcal P^{{\bf a}_{I_P}}_{\mu}(\lambda)$ of alcove walks (see below for the precise definition). The existence of cellular pavings was proved for objects attached to general parahoric affine flag varieties in \cite{Hai25}, but due to the general setting those pavings are not at all explicit. In this article, we give an essentially self-contained alternative treatment of the present special case of \cite{Hai25}, and in addition we make the result much more explicit, by identifying the index set of the paving with particular alcove walks. The special setting of this article has the advantage of connecting representation theory invariants such as branching multiplicities to this same set of alcove walks. 

Recall that specifying a paving of a scheme $X$ by schemes in a class $\mathcal C$ amounts to specifying a finite filtration by closed subschemes $\emptyset = X_0 \subset X_1 \subset \cdots \subset X_n = X$ such that each locally closed subscheme $X_i - X_{i-1}$ belongs to $\mathcal C$ (for $1 \leq i \leq n$). When describing a paving in this article, we often write it is a union of the locally closed subschemes which appear; we will leave unspecified the way to assemble them into closed subschemes to satisfy the definition of paving, because we are mainly interested in the form and enumeration of the pieces.

\begin{thmx}[Theorem \ref{Thm_A_body}]  \label{Thm_A}  Write $K = L^+G$ and $K_P = L^+M\,LN$. The \textup{(}reduced\textup{)} parabolic Mirkovi\'c-Vilonen intersection has a paving by $k$-schemes of the form
\begin{equation}  \label{Thm_A_eq}
K_P t^\lambda K/K \, \cap \, Kt^\mu K/K \cong \bigsqcup_{{\bf a}_\bullet \in \mathcal P^{{\bf a}_{I_P}}_{\mu}(\lambda)} \mathbb A_k^{c^+({\bf a}_\bullet)} \times (\mathbb A^1_k - \mathbb A^0_k)^{f^+({\bf a}_\bullet)}.
\end{equation} 
\end{thmx}
Here $\mu  \in X_*(T)^+$ is dominant for $G$ and $\lambda \in X_*(T)^{+_M}$ is dominant for $M$ (see $\S$\ref{alcove_walks_sec},\ref{alcove_walk_models_sec} for the remaining notation). An important feature is the precise set of alcove walks ${\bf a}_\bullet$ that appear:
$$
\mathcal P^{{\bf a}_{I_P}}_\mu(\lambda) =  \coprod_{w \in W_0/W_{0,-\mu}} \coprod_{w'' \in W_{M,0}/W_{M,0,-\lambda}} \mathcal P^{{\bf a}_{I_P}}_{(t_{-w(\mu)})_{\bf 0}}(-w''(\lambda)).
$$

Here $W_0$ is the finite Weyl group for $G$, and $(t_{-w(\mu)})_{\bf 0}$ is the unique right $W_0$-minimal element in the coset $t_{-w(\mu)}W_0$ in the extended affine Weyl group. The set $\mathcal P^{{\bf a}_{I_P}}_{(t_{-w(\mu)})_{\bf 0}}(-w''(\lambda))$ consists of the alcove walks in the base apartment $\mathcal A_T$ which are {\em positively-folded} with respect to any alcove ${\bf a}_{I_P}$ ``towards infinity''  in  a suitable direction (depending on $P$), which have type $(t_{-w(\mu)})_{\bf 0}$ (choose any reduced expression), and which start at the base alcove ${\bf a}$ and terminate at the vertex $-w''(\lambda) + {\bf 0}$.

In Proposition \ref{K_P_upper_bd_prop} we prove an a priori upper bound $\langle \rho, \mu + \lambda \rangle$ on the dimension of the left hand side of (\ref{Thm_A_eq}), which leads to the same upper bound on the dimension $c^+({\bf a}_\bullet) + f^+({\bf a}_\bullet)$ of each intervening alcove walk ${\bf a}_\bullet$.  Then defining $\mathcal M^{{\bf a}_{I_P}}_\mu(\lambda) \subset \mathcal P^{{\bf a}_{I_P}}_\mu(\lambda)$ as the subset of alcove walks of dimension $\langle \rho, \mu + \lambda \rangle$, we parametrize the irreducible components of maximum possible dimension, and deduce a relation to the branching to Levi subgroups on the side of the Langlands dual groups $\widehat{G} \supset \widehat{M} \supset \widehat{T}$.  This is our second main theorem:

\begin{thmx} [Theorem \ref{main_thm_P}] \label{Thm_B}
Let $\mu \in X_*(T)^+$ and $\lambda \in X_*(T)^{+_M}$. Then: 
\begin{enumerate}
\item[(1)] There is a bijection ${\bf a}_\bullet \mapsto C_{{\bf a}_\bullet}$ 
\begin{equation*} \label{bij_P_1}
\mathcal M^{{\bf a}_{I_P}}_\mu(\lambda) ~~ \overset{\sim}{\longrightarrow} ~~ {\rm Irred}^{\langle \rho, \mu + \lambda \rangle}(K_Pt^\lambda K/K \, \cap \, Kt^\mu K/K).
\end{equation*} 
\item[(2)] The multiplicity $[V^{\widehat{G}}_\mu \, : \, V^{\widehat{M}}_\lambda]$ is the cardinality of the set $\mathcal M^{{\bf a}_{I_P}}_\mu(\lambda)$ of alcove walks in ${\mathcal  P}^{{\bf a}_{I_P}}_\mu(\lambda)$ which have the maximum possible dimension $\langle \rho, \mu + \lambda \rangle$.
\end{enumerate}
\end{thmx}
Specialized to $P = B$, we get in Theorem \ref{main_thm_B} a parameterization of the irreducible components in $LU t^\lambda K/K\, \cap \, Kt^\mu K/K$, and a new algorithm for computing characters of the representations $V^{\widehat{G}}_\mu$.  An example of Theorem \ref{main_thm_B} is worked out in section $\S$\ref{A2_example}; see Figure \ref{fig:fig1} for a picture of the relevant alcove walks. We illustrate Theorem \ref{main_thm_P}  by proving in Proposition \ref{PRV_prop} a result on branching to Levi subgroups, which is analogous to (but easier than) the PRV Conjecture for tensor product multiplicities (comp.\,\cite{Kum90, Mat89}).

In $\S$\ref{variant_sec}, we present in Theorems \ref{Conv_AW_model}  and \ref{Tensor_AW_model} the analogues of Theorems \ref{Thm_A} and \ref{Thm_B} for convolution fibers and for the multiplicities in tensor products of highest weight representations. Finally, in $\S\ref{appendix}$, we give further information about the quantities which intervene in describing alcove walks.

\subsection{Relation to other works in the literature}

The main purpose of this paper is to provide a simple proof of an explicit cellular paving as in Theorem \ref{Thm_A}, in terms of alcove walks alone (comp.\,\cite{GL05}, which uses the more general galleries), {\em in a framework which applies not just to the classical Mirkovi\'c-Vilonen intersections \textup{(}the case $P = B$\textup{)} but equally well to the general parabolic versions thereof}.

The results of this paper are somewhat parallel to those of \cite{GL05}, although there is very little overlap in the methods of proof.  For the special case of $P =B$, Gaussent and Littelmann prove in \cite[Thm\,4 and Cor.\,5]{GL05} versions of our Theorems \ref{Thm_A} and \ref{Thm_B}, in the language of their {\em positively folded galleries}.  The positively-folded alcove walks above are special cases of their positively-folded galleries.  One could say that we are able to express the MV-cycles and cellular pavings solely using alcove walks, but the price we pay is that we have to include all the different types $(t_{-w(\mu)})_{\bf 0}$ for $w \in W_0/W_{0, -\mu}$. It would be of interest to give an explicit bijection between the galleries appearing in \cite{GL05} and the alcove walks appearing in the Theorems \ref{Thm_A} and \ref{Thm_B}. The author will pursue this in a future work.

Theorem \ref{main_thm_B} gives a new combinatorial parametrization of MV cycles, the irreducible components in an MV intersection in the case $P=B$, and Theorem \ref{Thm_B} gives information in the same direction for the cases where $P \neq B$; see Example \ref{branching_counterexample} to see that the geometry in the general situation really is much more subtle.  The MV cycles, of course, play a key role in many developments, first and foremost in the geometric Satake equivalence of Mirkovi\'c-Vilonen \cite{MV07}, and then in more recent extensions of that equivalence to novel situations; see, for example, \cite{RS21}, \cite{FS24}, \cite{BV25}, and the aforementioned \cite{CSvdH22}.

In \cite{Ram06}, there are several results which bear some relations to our theorems.  Namely, \cite[Thm.\,4.2, 4.10]{Ram06} are related, and essentially look like $q$-analogues of Theorem \ref{Thm_B}(2), and \cite[Thm.\,4.9]{Ram06} appears to be a $q$-analogue of Theorem \ref{Tensor_AW_model}(2). It would be interesting to give a detailed comparison linking Ram's results with the results of this article.

Finally, the results of Theorems  \ref{Conv_AW_model} and \ref{Tensor_AW_model} can be seen as purely alcove-walk analogues of the results of Kapovich-Leeb-Millson \cite{KLM08} on Hecke paths in the Bruhat-Tits building and their applications to tensor product multiplicities. It is not yet clear to the author whether the alcove-walk models shed any new light on the generalized saturation conjectures they were concerned with.

 In the same way, Theorem \ref{Thm_B}(2), Theorem \ref{main_thm_B}(2), and Theorem \ref{Tensor_AW_model}(2) can be seen as giving purely alcove-walk versions of some results of Littelmann expressed using the Littelmann Path Model \cite{Li95}.

\medskip

\noindent {\bf Acknowledgments:} I thank Ulrich G\"{o}rtz, João Lourenço, Shrawan Kumar, Elizabeth Milicevi\'{c}, Sian Nie, and Thibaud van den Hove for very helpful comments, suggestions,  and questions about a preliminary version of this article. I owe a special debt to Elizabeth Milicevi\'c for teaching me how to make pictures such as the one in Figure \ref{fig:fig1}. I thank Jeffrey Adams for his crucial help with Example \ref{branching_counterexample}. I am especially grateful to Wenzhuo Wang for discovering an error in an earlier version of this article. I am grateful to the anonymous referees for their helpful comments and suggestions.

\section{Notation}

All $k$-schemes and $k$-ind-schemes in this paper are endowed with their reduced structure, in particular this applies to all intersections, unions, and unions given by pavings or stratifications.

\subsection{Group theory notation}
In the group $G$, we fix a $k$-torus $T$ and $k$-rational Borel subgroup $B = TU$, where $U$ is the unipotent radical of $B$. We also consider standard parabolic subgroups $P = MN \supset B$, where $M \supset T$ is a Levi factor and $N$ is the unipotent radical of $P$. Let $\langle \cdot \,  , \, \cdot \rangle : X^*(T) \times X_*(T) \rightarrow \mathbb Z$ be the canonical perfect pairing.  The ($B$-)positive roots $\Phi^+$ in the set of roots $\Phi(G, T) \subset X^*(T)$ are defined to be those appearing in ${\rm Lie}(B)$. Let $\rho \in X^*(T)$ denote the half-sum of the $B$-positive roots.

Let $F = k\rpot{t}$, the field of Laurent series in the variable $t$, with ring of integers $\mathcal O = k\pot{t}$. In the (enlarged) Bruhat-Tits building $\mathfrak B(G,F)$, we fix the apartment $\mathcal A = \mathcal A_T$ corresponding to $T$, and we choose an origin and identify $\mathcal A \cong X_*(T)  \otimes \mathbb R =: V$. The affine hyperplanes in $\mathcal A$ determine a Coxeter complex in $\mathcal A$ and are the zero-sets of the affine roots $\Phi_{\rm aff} = \Phi \times \mathbb Z$; we write $(\alpha, n) \in \Phi_{\rm aff}$ as $\alpha + n$ and identify it with the affine-linear $\mathbb R$-functional $v \mapsto \alpha(v) + n$ for $v \in V$. We define the {\em dominant Weyl chamber} $\mathcal C := \{ v \in V \, | \, \langle \alpha, v \rangle > 0, \,\,\forall \alpha \in \Phi^+\}$; write $X_*(T)^+ = X_*(T) \cap \overline{\mathcal C}$, where $\overline{\mathcal C}$ is the closure of $\mathcal C$ in $V$. We define the {\em base alcove} ${\bf a}$ to be the unique alcove contained in $\mathcal C$ whose closure contains the origin.  Then the positive affine roots are characterized as 
$$
\Phi^+_{\rm aff} = \{ \alpha + n \, | \, \alpha + n \,\, \mbox{is positive on ${\bf a}$} \}.
$$

\subsection{Weyl group notation and conventions} The sets $\Phi^+$ and $\Phi^+_{\rm aff}$ contain subsets $\Delta$ and $\Delta_{\rm aff}$ of minimally-positive elements, and these give rise to the simple reflections $S$ and the simple affine reflections $S_{\rm aff}$, all acting on $V$.  The affine Weyl group is the Coxeter group $W_{\rm aff} = \langle s \, | \, s \in S_{\rm aff}\rangle$, viewed as a subgroup of  ${\rm Aut}(V)$, the group of affine-linear automorphisms of $V$.  The finite Weyl group is the Coxeter subgroup fixing the origin, $W_0 = \langle s \, | \, s \in S\rangle$.  The extended affine Weyl group is the group $W = X_*(T) \rtimes W_0 \subset {\rm Aut}(V)$, where $\lambda \in X_*(T)$ acts by translation by $\lambda$, denoted $t_\lambda$. The action permutes the affine root hyperplanes and thus the alcoves; let $\Omega = \Omega_{\bf a} \subset W$ denote the stabilizer of ${\bf a}$.  The group $W_{\rm aff}$ acts simply transitively on the set of alcoves. Thus we obtain a decomposition $W = W_{\rm aff} \rtimes \Omega$.  We extend the Bruhat order $\leq$ and the length function $\ell(\cdot)$ on $W_{\rm aff}$ to $W$ using this decomposition: $\ell(w\tau) = \ell(w)$ for $w \in W_{\rm aff}$ and $\tau \in \Omega$.  Morever $w\tau \leq w' \tau'$ if and only $\tau = \tau' $ and $w \leq w'$ in $W_{\rm aff}$. The group $\Omega$ is finitely-generated and abelian, and we consider it as the subgroup of length-zero elements in $W$.

Let $N_G(T)$ denote the normalizer of the torus $T$ in $G$. Given our choice of origin, the Iwahori-Weyl  group $\mathcal W := N_G(T)(F)/T(\mathcal O)$ acts on $V$ in a natural way, but there is a choice of convention here which is important. Given $\lambda \in X_*(T)$, define $t^\lambda := \lambda(t) \in T(F)$ and use the same symbol for its image in $\mathcal W$. The lifts $\dot{w}$ are declared to act on $V$ as do their images $w \in W_0$ (fixing the origin). The action of $t^\lambda$ is given by the Bruhat-Tits convention: {\bf it acts by $t_{-\lambda}$}.  Then we see easily that we get an isomorphism $\mathcal W \cong W$, whereby $t^\lambda \in \mathcal W$ is identified with $t_{-\lambda} \in W$. The resulting action of $N_G(T)(F)$ on $\mathcal  A$ extends naturally to an action of $G(F)$ on the building $\mathfrak B(G, F)$, and this action is fixed once and for all.

For $s = s_{\alpha + n}$, we denote the corresponding affine hyperplane by $H_s = H_{\alpha + n} = \{ v \in V \, | \, \alpha(v) + n = 0\}$.  Each wall of ${\bf a}$ lies in $H_s$, for a unique $s \in S_{\rm aff}$.

 \subsection{Loop group, Iwahori and parahoric subgroup notation}

For any facet ${\bf f} \subset \mathcal A$ let $\mathcal G_{\bf f}$ denote the corresponding Bruhat-Tits parahoric group scheme over $\mathcal O$ with connected geometric fibers; see \cite{HR08} for a characterization of the parahoric subgroup $\mathcal G_{\bf f}(\mathcal O) \subset G(F)$ as the subgroup of the Kottwitz kernel which fixes ${\bf f}$ pointwise. Let ${\bf 0}$ be the unique facet in the closure of ${\bf a}$ containing the origin (we will often call such minimal facets ``vertices''). Regard the split form $G_{\mathcal O}$ as the Bruhat-Tits hyperspecial maximal parahoric group scheme $\mathcal G_{\bf 0}$, and let $\mathcal G_{\bf a}$ denote the Iwahori group scheme for ${\bf a}$.  

We use the standard loop group functors on $k$-Algebras: $LG(R) = G(R\rpot{t})$, and $L^+\mathcal G_{\bf f}(R) = \mathcal G_{\bf f}(R\pot{t})$. We abbreviate by setting $K = L^+G = L^+\mathcal G_{\bf 0}$, and $I = L^+\mathcal G_{\bf a}$.

\subsection{Union, paving, and stratification notation}

Throughout the article, an equality such as $\sqcup_i Z_i = Z$ designates a finite (or possibly countable) paving (or sometimes even a stratification) of a scheme or ind-scheme $Z$ into disjoint locally closed subschemes $Z_i$ (the underlying sets are disjoint but there could be closure relations between them -- they are not necessarily coproducts in the category of topological spaces). The symbol $\coprod$ denotes the disjoint union of sets.

\subsection{Affine flag variety notation}

We work with the affine Grassmannian 
$${\rm Gr}_G = LG/L^+G = LG/K$$ 
(the \'etale sheafification of the presheaf $R \mapsto LG(R)/L^+G(R)$), which is representable by an ind-projective ind-scheme over $k$. Similarly we have the affine flag variety ${\rm Fl}_G = (LG/L^+{\mathcal G}_{\bf a})^{\mbox{\small \rm \'{e}t}} = (LG/I)^{\mbox{\small \rm \'{e}t}}$.  

Fix a standard parabolic subgroup $P = MN$. Taking intersections in $LG$, we have $LM \cap I= L^+\mathcal M_{{\bf a}_M}$, where $\mathcal M_{{\bf a}_M}$ is the unique Iwahori group scheme attached to $M$ and the unique alcove ${\bf a}_M \in \mathfrak B(M, F)$ whose image in $\mathfrak B(G,F)$ contains ${\bf a}$ (see \cite[Lem.\,2.9.1]{Hai09}).

We define group ind-schemes $K_P = L^+M \, LN = K_M LN$, and $I_P = L^+\mathcal M_{{\bf a}_M} LN = I_M LN$ which act on the left on ${\rm Gr}_G$ and on ${\rm Fl}_G$.  We have {\em stratifications} (the Cartan and Bruhat-Tits decompositions) 

\begin{equation}
LG = \bigsqcup_{\mu \in X_*(T)^+} K t^\mu K \hspace{.25in}, \hspace{.25in} LG = \bigsqcup_{w \in W} Iw I = \bigsqcup_{\nu \in X_*(T)} I t^\nu K  \hspace{.25in} , \hspace{.25in} K = \bigsqcup_{w \in W_0} IwI
\end{equation}
and the corresponding stratifications in terms of orbits of ${\rm Gr}_G$ and ${\rm Fl}_G$. We also have the usual BN-pair relations for $(I, W, S)$ and from these we deduce the following closure relations in ${\rm Fl}_G$:
$$
\overline{I w I/I} = \bigsqcup_{v \leq w} IvI/I.
$$
\begin{rmk} \label{minus_rmk}
When we wish to emphasize the role played by the Bruhat order and closure relations, we often write $It_{-\lambda}I/I$ (resp.\,$Kt_{-\lambda}K/K$) in place of $It^\lambda I/I$ (resp.\,$Kt^\lambda K/K$), etc. It is important for the closure relations that we use the conventions we established above.  For example, if $\lambda \in X_*(T)^+$, then for a simple reflection $s_\alpha \in S$ ($\alpha \in \Delta$)  we have the BN-pair relation $I s_\alpha I t^\lambda I = Is_\alpha I t_{-\lambda}I \subseteq I t_{-\lambda} I \sqcup I s_\alpha t_{-\lambda} I =  I t^{\lambda} I \sqcup I s_\alpha t^{\lambda} I$.
\end{rmk}

\section{An upper bound on dimensions}

In this section, fix  $\mu \in X_*(T)^+$ and $\lambda \in X_*(T)^{+_M}$ (the latter being the cocharacters dominant with respect to the positive roots in $\Phi(M, T)$). Let $\widehat{G} \supset \widehat{M} \supset \widehat{T}$ be the Langlands duals of $G \supset M \supset  T$ over $\mathbb C$.  Let $V^{\widehat{G}}_\mu$ be the irreducible representation of $\widehat{G}$ with highest weight $\mu \in X^*(\widehat{T})^+ = X_*(T)^+$, and let $[V^{\widehat{G}}_\mu : V^{\widehat{M}}_\lambda]$ denote the multiplicity of $V^{\widehat{M}}_\lambda$ in the restriction of $V^{\widehat{G}}_\mu$ to $\widehat{M}$. 

\begin{dfn} For a finite-type reduced $k$-scheme  $X$, let ${\rm Irred}(X)$ denote the finite set of irreducible components of $X$, and for any $d \in \mathbb Z_{\geq 0}$ let ${\rm Irred}^{d}(X)$ denote the subset of $d$-dimensional irreducible components of $X$.
\end{dfn}

\begin{prop} \label{K_P_upper_bd_prop}
In the above situation, we have the inequality
\begin{equation} \label{K_P_inequality_eq}
{\rm dim}\big(K_P t^\lambda K/K \, \cap \, K t^{\mu} K/K\big) ~ \leq ~ \langle \rho, \mu + \lambda \rangle
\end{equation}
and moreover $\#{\rm Irred}^{\langle \rho, \mu + \lambda \rangle}(K_P t^\lambda K/K \, \cap \, K t^{\mu} K/K) = [V^{\widehat{G}}_\mu :V^{\widehat{M}}_\lambda]$.
\end{prop}

\begin{proof}
This result is well known.  For example, it follows from \cite[Prop.\,5.3.29]{BD} and from  \cite[Thm.\,2.2]{BG01}.  In both of those sources, the argument relies on the geometric Satake isomorphism of Mirkovi\'c-Vilonen \cite{MV07} (and technically those sources assume the field $k$ is $\mathbb C$, or more generally a characteristic zero field). Our purpose here is merely to point out that this result, including the statement about branching multiplicities, is more elementary and does not in fact require the geometric Satake isomorphism, but only its precursor, the Lusztig-Kato formula.

When $k$ is a finite field $\mathbb F_q$ or its algebraic closure, the proposition is proved by combining \cite[Prop.\,5.4.2]{GHKR06} with \cite[Lem\,9.3]{HKM}.  Ultimately this relies on the Lusztig-Kato formula, see \cite{Lu81,kato} and \cite{HKP}.

Suppose $q$ is a power of a prime $p$. If $k$ is such a field and $W(k)$ is its ring of $p$-typical Witt vectors with fraction field $K(k)$, then we may define $W(k)$-integral versions $\underline{G}, \underline{B},\underline{T},\underline{U}, \underline{W}, \underline{\bf f}, \underline{\mathcal G}_{\underline{\bf f}}, \underline{K}, \underline{I}$ of $G,B,T,U, W,\mathcal G_{\bf f}, K,I$ and thus we can consider the $\underline{P}$-versions of the Mirkovi\'c-Vilonen intersections in $L\underline{G}/L^+\underline{\mathcal G}_{\underline{\bf f}}$ (see \cite[$\S$6]{HaRi20}).  The formation of Mirkovi\'c-Vilonen intersections commutes with base change along $W(k) \rightarrow k$, so that in particular
$$
\Big(\underline{K}_{\underline{P}} t^\lambda \underline{K}/\underline{K} \, \cap\, \underline{K} t^\mu \underline{K}/\underline{K}\Big) ~ \otimes_{W(k)} ~ k ~~ \cong ~~ K_P t^{\lambda} K/K \, \cap K t^\mu K/K
$$
where the right hand side designates the objects over $k$.  There is a similar identity for base-changing to the fraction field $K(k)$ of $W(k)$.  By flatness of all objects over $W(k)$, the dimension estimate and description of top-dimensional irreducible components over $k$ implies the corresponding results over $K(k)$.  Now since $K(k)$ is characteristic zero, a standard argument reduces any characteristic zero field of coefficients to $K(k)$.
\end{proof}

\section{Alcove walks} \label{alcove_walks_sec}

\subsection{Alcove walks and their labels}

Let ${\bf a}$ denote the base alcove.  For any alcove ${\bf b} \subset \mathcal A_T$, there is a unique element $w_{\bf b} \in W_{\rm aff}$ such that ${\bf b} = w_{\bf b} ({\bf a})$. Let $s$ denote a simple affine reflection, that is, a reflection through one of the walls of ${\bf a}$. If ${\bf b}$ and ${\bf c}$ share a face of type $s$, then the crossing from ${\bf b}$ to ${\bf c}$ over the type $s$-hyperplane $w_{\bf b}H_s$, denoted ${\bf b} \overset{s}{\rightarrow} {\bf c}$, is achieved by applying the reflection $w_{\bf b} s w_{\bf b}^{-1}$ to ${\bf b}$.

\begin{dfn} Let ${\bf b}$, ${\bf c}$, ${\bf d}$ be three alcoves in the apartment ${\mathcal A}_T$ and suppose ${\bf c}$ and ${\bf d}$ are distinct and share a codimension 1 face of type $s \in S_{\rm aff}$.  We say that the crossing ${\bf c} \overset{s}{\rightarrow} {\bf d}$ is {\em in the ${\bf b}$-positive direction} if ${\bf b}$ and ${\bf c}$ are on the same side of the affine hyperplane $w_{\bf c}H_s$. Similarly, we say the crossing is {\em in the ${\bf b}$-negative direction} if ${\bf b}$ and ${\bf d}$ are on the same side of $ w_{\bf c}H_s$.
\end{dfn}

\begin{dfn} Let ${\bf b}$ denote any fixed alcove in the apartment $\mathcal A_T$, and fix a word (not necessarily reduced) $s_\bullet = s_1 s_2 \cdots s_r \tau \in W$ with $\tau \in \Omega$.
We say ${\bf a}_0, {\bf a}_1, \dots, {\bf a}_r$ is a {\em ${\bf b}$-positively folded alcove walk of type $s_\bullet$} if for all $1 \leq i \leq r$, the following conditions hold, writing $w_{i} := w_{{\bf a}_i}$ for all $i$:
\begin{enumerate}
\item[(0)] ${\bf a}_i \in \{ {\bf a}_{i-1}, w_{i-1}s_iw_{i-1}^{-1}{\bf a}_{i-1}\}$
\item[(1)] if ${\bf a}_{i-1} = {\bf a}_i$, the crossing ${\bf a}_{i-1} \overset{s_i}{\rightarrow} w_{i-1}s_iw_{i-1}^{-1}{\bf a}_{i-1}$ is in the ${\bf b}$-negative direction.
\end{enumerate}
If ${\bf a}_{i-1} \neq {\bf a}_i$ we call ${\bf a}_{i-1} \overset{s_i}{\rightarrow} {\bf a}_i$ a {\em crossing} (which may be either in the ${\bf b}$-positive or in the ${\bf b}$-negative direction). If ${\bf a}_{i-1} = {\bf a}_i$, we still denote the process as ${\bf a}_{i-1} \overset{s_i}{\rightarrow} {\bf a}_{i}$ and we call it a {\em folding} ({\em in the ${\bf b}$-positive direction}). 
\end{dfn}

We now introduce an alternate notation for the operation  ${\bf a}_{i-1} \overset{s_i}{\rightarrow} {\bf a}_i$.

\begin{dfn} When ${\bf b}$ is understood, we denote crossings and foldings as follows:
\begin{enumerate}
\item[(i)] If ${\bf a}_{i-1} \overset{s_i}{\rightarrow} {\bf a}_i$ is a crossing in the ${\bf b}$-positive direction, we write it as ${\bf a}_{i-1} \overset{c^+_{s_i}}{\rightarrow} {\bf a}_i$.
\item[(ii)] If ${\bf a}_{i-1} \overset{s_i}{\rightarrow} {\bf a}_i$ is a crossing in the ${\bf b}$-negative direction, we write it as ${\bf a}_{i-1} \overset{c^-_{s_i}}{\rightarrow} {\bf a}_i$.
\item[(iii)] If ${\bf a}_{i-1} \overset{s_i}{\rightarrow} {\bf a}_{i}$ is a folding, we write it as ${\bf a}_{i-1} \overset{f^+_{s_i}}{\rightarrow} {\bf a}_i$ or ${\bf a}_{i-1} \overset{f^+_{s_i}}{\rightarrow} {\bf a}_{i-1}$.
\end{enumerate}
\end{dfn}

\begin{dfn} Assume the alcove ${\bf b}$, which determines the orientations of the crossings and foldings, is understood. A {\em labeled \textup{(}positively-folded\textup{)} alcove walk} is a positively-folded alcove walk ${\bf a}_\bullet = ({\bf a}_0, {\bf a}_1, \dots, {\bf a}_r)$, with each crossing or folding ${\bf a}_{i-1} \overset{s_i}{\rightarrow} {\bf a}_i$ labeled with the symbol $c^+_{s_i}, c^-_{s_i}$, or $f^+_{s_i}$ which describes it.
\end{dfn}

\begin{dfn}
Let ${\bf b}$ be any alcove in the apartment $\mathcal A_T$. For $x, y \in W$ and any reduced word expression $\dot{x} = s_1 \cdots s_r \tau$ for $x$, let $\mathcal P^{{\bf b}}_{\dot{x}}(y)$ denote the set of ${\bf b}$-positively folded labeled alcove walks of type $x_\bullet$ from ${\bf a}$ to $y{\bf a}$.
\end{dfn}
In practice, we often suppress the dot in the notation $\dot{x}$. It is always implicit that both $x$ and $y$ belong to the same coset $W_{\rm aff} \rtimes \tau$.

\subsection{Dimension of a (labeled) alcove walk}

\begin{lem} For a labeled positively-folded alcove walk ${\bf a}_\bullet$ as above, we have $$r= f^+({\bf a}_\bullet) + c^+({\bf a}_\bullet) + c^-({\bf a}_\bullet),$$ 
where
\begin{itemize}
\item $f^+({\bf a}_\bullet) = \#\{ i \, | \, {\bf a}_{i-1} \overset{s_i}{\rightarrow} {\bf a}_i = {\bf a}_{i-1} \,\,\,\,\mbox{is a folding \textup{(}necessarily positive\textup{)}}\}$.
\item $c^+({\bf a}_\bullet) = \#\{ i \, | \, {\bf a}_{i-1} \overset{s_i}{\rightarrow} {\bf a}_i \,\,\,\,\mbox{is a crossing in the ${\bf b}$-positive direction}\}$.
\item $c^-({\bf a}_\bullet) = \#\{ i \, | \, {\bf a}_{i-1} \overset{s_i}{\rightarrow} {\bf a}_i \,\,\,\,\mbox{is a crossing in the ${\bf b}$-negative direction}\}$.
\end{itemize}
\end{lem}

\begin{dfn} We define the dimension of a ${\bf b}$-positively folded labeled alcove walk ${\bf a}_\bullet$ to be
$$
{\rm dim}({\bf a}_\bullet) = c^+({\bf a}_\bullet) + f^+({\bf a}_\bullet).
$$
\end{dfn}

\begin{rmk}
Note that a given alcove walk ${\bf a}_{\bullet}$ in $\mathcal A_T$ may be positively-folded with respect to one choice of ${\bf b}$ while not for another choice ${\bf b}'$.  The set of labels depends on the choice of ${\bf b}$ as well.  However, for any two ${\bf b}, {\bf b'}$ sufficiently deep inside a Weyl chamber, ${\bf a}_\bullet$ will be ${\bf b}$-positively folded if and only if it is ${\bf b}'$-positively folded, and the set of labels, and thus the dimensions, determined by ${\bf b}$ and ${\bf b}'$ will coincide (see Lemma \ref{I_P_vs_retraction} below).
\end{rmk}

\section{Alcove walk models for parabolic Mirkovi\'c-Vilonen intersections} \label{alcove_walk_models_sec}

\subsection{Retractions and $I_P$-orbits}

Fix $P = MN$ as before. We recall $K_P = LN K_M$ and $I_P = LN I_M$.  Here $K_M = LM \cap K$ and $I_M = I \cap LM$.  Let ${\bf a}_M \supseteq {\bf a}$ denote the choice of base alcove for the apartment in $\mathfrak B(M, F)$, lying between the affine hyperplanes $H_{\alpha}$ and $H_{\alpha -1}$ for all the positive roots $\alpha$ appearing in ${\rm Lie}(M)$. Then $I_M$ is the Iwahori group for $(M, {\bf a}_M)$ (\cite[Lem.\,2.9.1]{Hai09}).

\begin{dfn}
We let $\nu \in X_*(T)^+$ be a dominant cocharacter. Consider the properties:
\begin{itemize}
\item $\nu$ is $M$-central, that is, $\langle \alpha, \nu \rangle = 0$ for all $\alpha \in \Phi(M,T)$. 
\item $\nu$ has the property that $\langle \alpha, \nu \rangle >\!>0$ for all roots $\alpha$ appearing in ${\rm Lie}(N)$.
\end{itemize}
When $\nu$ satisfies the first property, we say it is {\em $M$-central}. When $\nu$ satisfies the second property, we say it is {\em sufficiently $N$-dominant}.
\end{dfn}

\begin{rmk}
The notation $\langle \alpha, \nu \rangle >\!> 0$ means that these numbers are {\em large, in a way which can in principle be specified depending on the problem at hand}. The results are always of this form: {\em Given a certain finite problem, a desired property holds if we require all the numbers $\langle \alpha, \nu \rangle$ are at least $m$ \textup{(}where $m \in \mathbb Z_{\geq 0}$  depends on the finite problem\textup{)}}.
\end{rmk}

For each alcove ${\bf b} \subset \mathcal A$, there is a retraction $\rho_{{\bf b}, \mathcal A} : \mathfrak B(G,F) \rightarrow \mathcal A$, which is simplicial and distance-preserving from the alcove ${\bf b}$.  

\begin{lem} \label{I_P_vs_retraction}
Suppose $I_P$ is defined using $P$ and ${\bf a}_M$ as above. For any bounded subset $\mathcal Y \subset \mathfrak B(G,F)$, there exists a bounded subgroup $I_{\mathcal Y} \subset I_P$ with the properties:
\begin{enumerate}
\item[(i)] $\mathcal Y \subset \cup_{g \in I_{\mathcal Y}} g^{-1}\mathcal A$,
\item[(ii)] for all alcoves ${\bf b} = t_{\nu}{\bf a}$ for $\nu$ $M$-central and sufficiently $N$-dominant such that the connected fixer subgroups satisfy $I_{\mathcal Y} \subset I_{\bf b}$, the retractions $\rho_{{\bf b}, \mathcal A}$ all agree on $\mathcal Y$.
\end{enumerate}
Further, defining the retraction $\rho_{I_P, \mathcal A}: \mathfrak B(G, F) \rightarrow \mathcal A$ as the common value of $\rho_{{\bf b}, \mathcal A}$ for ${\bf b}$ defined as in \textup{(}ii\textup{)} on any given bounded subset, we have for every alcove ${\bf x} \subset \mathcal A$,
$$
\rho_{I_P, \mathcal A}^{-1}({\bf x}) = I_P {\bf x}.
$$
\end{lem}

\begin{proof}
See \cite[Lem.\,6.4, Lem.\,6.5]{HKM}.  These results elaborated on material from \cite[$\S$11.2]{GHKR10}, where these particular retractions were first studied.
\end{proof}
In what follows we write $\rho_{I_P, \mathcal A}$ as $\rho_{{\bf a}_{I_P}, \mathcal A}$ where ${\bf a}_{I_P}$ is any ${\bf b}$ as in (ii).
A final ingredient we need is the following. We know that for any reduced expression $\dot{x} = s_1 \cdots s_r \tau \in W$, the $k$-variety $IxI/I$ may be identified with an {\em uncompactified Demazure variety over $k$} in (a given connected component of) a twisted product ${\rm Fl}_G \widetilde{\times} \cdots \widetilde{\times} {\rm Fl}_G$, as in \cite[$\S$7.1]{Hai25}. The following result plays a key role in this article.

\begin{lem} \label{base_case_decomp}
For any $x, y \in W$ with the same $\Omega$-component, and fixing a reduced expression $\dot{x}$ for $x$, we have a paving in the category of $k$-schemes with reduced structure
$$
I_P y I/I \cap I x I/I \cong \bigsqcup_{{\bf a}_\bullet \in \mathcal P^{{\bf a}_{I_P}}_{x}(y)} \mathbb A_k^{c^+({\bf a}_\bullet)} \times (\mathbb A_k^1 - \mathbb A_k^0)^{f^+({\bf a}_\bullet)}.
$$
\end{lem}
The right hand side is describing a cellular paving of the left hand side, in the sense of \cite{Hai25}. The left hand side is given reduced structure, and the terms in the union on the right are locally closed subschemes.  We make no claim about the closure relations.

\begin{proof}
First we use Lemma \ref{I_P_vs_retraction} to interpret the $I_P$-orbit in terms of retraction with respect to any suitable single alcove ${\bf a}_{I_P}$. Then the statement is a direct consequence of well-known properties of retractions, see \cite[Prop.\,2.3.12, Prop.\,2.9.1]{BT72}.  A summary of how this works is given in \cite[$\S$6.1]{GHKR06}.  One can also use the method of proof in \cite[$\S$6.3]{Hai25}.
\end{proof}

\subsection{A description of parabolic Mirkovi\'c-Vilonen intersections}

We fix $\mu \in X_*(T)^+$ and $\lambda \in X_*(T)^{+_M}$.  We let $\nu \in X_*(T)^+$ be an auxillary dominant cocharacter.

\begin{prop} \label{HKM_prop}
In the above situation, for $\nu$ $M$-central and sufficiently $N$-dominant, we have
$$
K_P t^\lambda K/K \, \cap \, K t^{\mu} K/K = (t^{-\nu}Kt^{\nu}) t^{\lambda}K/K \, \cap \, Kt^\mu K/K.
$$
\end{prop}

\begin{proof}
This is proved in  \cite[Prop.\,10.1]{Hai25}. Note that the reference \cite{HKM} cited there used the opposite convention for the action of $t^\nu$ on the apartment $\mathcal A_T$ (translation by $\nu$ instead of $-\nu$), but the proof still applies.
\end{proof}

For $\nu$ sufficiently $N$-dominant and $M$-central, we have the following pavings given up to equality or isomorphism (bear in mind Remark \ref{minus_rmk}):
\begin{align*}
K_P t^{\lambda} K/K \, \cap \, K t^{\mu} K/K &= t^{-\nu}K t^{\nu}t^{\lambda}K/K  \, \cap \, K t^{\mu} K/K \\
&\cong \bigsqcup_{w \in W_0/W_{0,-\mu}} \bigsqcup_{w''\in W_0/W_{0,-\nu - \lambda}} I w''t_{-\nu - \lambda} K/K \, \cap \, t_{-\nu} I wt_{-\mu} K/K.
\end{align*}
For $\nu$ sufficiently $N$-dominant, the only $w''$ which give rise to a nonempty intersection satisfy $w'' \in W_{0,M}$, or equivalently, $w''$ fixes $\nu$. To see this, note that the $k$-points of $I t_{-w''(\nu-\lambda)} K/K$ can be identified with the set of minimal facets (`vertices'')  in the building in the $LG(k)$-orbit of ${\bf 0}$, which retract to the vertex $-w''(\nu + \lambda) + {\bf 0}$ under the retraction $\rho_{{\bf a}, \mathcal A_T}$; if $w'' \notin W_{0,M}$, then $-w''(\nu + \lambda) + {\bf 0}$  can be made arbitrarily far from the vertex $-\nu + {\bf 0}$ by choosing $\nu$ sufficiently $N$-dominant, and in particular, it will not be among the vertices contained in the set $\rho_{{\bf a}, \mathcal A_T} (t_{-\nu} I t_{-w(\mu)} K/K)$. (It might be useful to note that the latter set is contained in a ball centered at $-\nu + {\bf 0}$  with a radius which may be bounded above independently of $\nu$; indeed, the diameter of the set $I wt_{-\mu}K/K$ is independent of $\nu$, and this diameter does not increase under translation and retraction.) Therefore the intersection corresponding to $w''$ can be nonempty only if $w''(\nu) = \nu$, that is, if $w'' \in W_{0,M}$.  Thus the above is
$$
\bigsqcup_{w \in W_0/W_{0,-\mu}}  \bigsqcup_{w'' \in W_{0,M}/W_{0,M,-\lambda}} It_{-\nu - w''(\lambda)}K/K \, \cap \, t_{-\nu} I wt_{-\mu} K/K.
$$
By \cite[Lem.\,7.2]{Hai25}, this is isomorphic to 
$$
\bigsqcup_{w \in W_0/W_{0,-\mu}} \bigsqcup_{w' \in W_0}\bigsqcup_{w'' \in W_{0,M}/W_{0,M,-\lambda}} I t_{-\nu} t_{-w''(\lambda)} w' I/I \, \cap \, t_{-\nu} I (w t_{-\mu})_{\bf 0} I/I.
$$
Recall that $(t_{-w(\mu)})_{\bf 0}$ denotes the unique right $W_0$-minimal element in the coset $t_{-w(\mu)}W_0$. Multiplying on the left by $t^{-\nu}$ and using that $t^{-\nu} I t^{\nu}$ approximates $I_P$ as in Lemma \ref{I_P_vs_retraction}, this is isomorphic to
$$
\bigsqcup_{w \in W_0/W_{0,-\mu}} \bigsqcup_{w' \in W_0} \bigsqcup_{w'' \in W_{0,M}/W_{0,M,-\lambda}}  I_P t_{-w''(\lambda)} w' I/I \, \cap \, I (w t_{-\mu})_{\bf 0} I/I.
$$
We have proved the first isomorphism in the following lemma, and the second follows from another application of \cite[Lem.\,7.2]{Hai25}.

\begin{lem} \label{key_P-MV_lem}
For $\mu \in X_*(T)^+$ and $\lambda \in X_*(T)$, there are pavings of $k$-schemes with reduced structure
\begin{align} 
K_Pt^\lambda K/K \, \cap \, Kt^\mu K/K &\cong
\bigsqcup_{w \in W_0/W_{0,-\mu}} \bigsqcup_{w' \in W_0} \bigsqcup_{w'' \in W_{0,M}/W_{0,M,-\lambda}}  I_P t_{-w''(\lambda)} w' I/I \, \cap \, I (w t_{-\mu})_{\bf 0} I/I \label{refined_union_I}\\ 
&\cong  \bigsqcup_{w \in W_0/W_{0,-\mu}} \bigsqcup_{w'' \in W_{0,M}/W_{0,M,-\lambda}}  I_P t_{-w''(\lambda)} K/K \, \cap \, I (w t_{-\mu})_{\bf 0} K/K. \label{refined_union_K}
\end{align}
\end{lem}

\subsection{More about the elements $(t_{-w(\mu)})_{\bf 0}$}

The following gives us a better understanding of the elements $(t_{-w(\mu)})_{\bf 0}$ which appear above.  Fix $\mu \in X_*(T)^+$, and let $W_{0,\mu} \subset W_0$ denote the fixer subgroup, which is a Coxeter subgroup generated by the simple finite reflections which fix $\mu$.  Let $W_0^{\mu} \subset W_0$ denote the set of elements $v \in W_0$ which are of  minimal length in their cosets $vW_{0,\mu}$. 

\begin{lem} \label{type_elements}
If $w \in W_0^{\mu}$, then $(t_{-w(\mu)})_{\bf 0} = t_{-w(\mu)} w$.  In particular,
$$
\ell((t_{-w(\mu)})_{\bf 0}) = \ell(t_{-\mu}) - \ell(w).
$$
\end{lem}

\begin{proof}
We need to show that for any positive root $\alpha$ with reflection $s_\alpha$, we have $x < x s_\alpha$ in the Bruhat order if $x = t_{-w(\mu)}w$. But $x < xs_\alpha$ if and only if $x\alpha = \alpha \circ x^{-1}(-)$ takes positive values on ${\bf a}$, which holds if and only if
\begin{equation} \label{positive_values}
\langle \alpha, \mu \rangle + \langle w\alpha, {\bf a} \rangle \subset (0, \infty).
\end{equation}

\noindent {\bf Case 1:} $\langle \alpha, \mu \rangle \geq 1$.  Then (\ref{positive_values}) holds, using that $\langle w\alpha, {\bf a} \rangle \subset (-1,1)$.

\noindent {\bf Case 2:} $\langle \alpha, \mu \rangle  = 0$.  In this case $s_\alpha \in W_{0, \mu}$, and then we have $w < ws_\alpha$, or equivalently, $w\alpha > 0$. This implies that (\ref{positive_values}) holds.
\end{proof}

\subsection{Alcove walk model for parabolic Mirkovi\'c-Vilonen intersections}

\begin{dfn} \label{P_b_mu_dfn}
Suppose ${\bf b} \subset \mathcal A_T$ be any alcove,  and fix $\mu \in X_*(T)^+$ and $\lambda \in X_*(T)^{+_M}$. We define:
$$
\mathcal P^{{\bf b}}_\mu(\lambda) =  \coprod_{w \in W_0/W_{0,-\mu}} \coprod_{w' \in W_0} \coprod_{w'' \in W_{M,0}/W_{M,0,-\lambda}} \mathcal P^{{\bf b}}_{(wt_{-\mu})_{\bf 0}}(t_{-w''(\lambda)}w').
$$
\end{dfn}

For each $w \in W_0/W_{0,-\mu}$ and each $w'' \in W_{M,0}/W_{M,0,-\lambda}$, we think of 
$$
\mathcal P^{\bf b}_{(t_{-w(\mu)})_{\bf 0}}(-w''(\lambda)) := \coprod_{w' \in W_0} \mathcal P^{{\bf b}}_{(wt_{-\mu})_{\bf 0}}(t_{-w''(\lambda)}w')
$$
as the set of alcove walks from the base alcove ${\bf a}$ to the vertex $-w''(\lambda) + {\bf 0}$, having type given by $(t_{-w(\mu)})_{\bf 0}$, such that the alcove walk is ${\bf b}$-positively folded. An alcove walk terminating in a vertex is by definition one that terminates in any alcove which has that vertex in its closure.  If the vertex is $-w''(\lambda) + {\bf 0}$, then the set of alcoves sharing that vertex is precisely $\{ t_{-w''(\lambda)} w'{\bf a} \, | \, w' \in W_0\}$.

Now fix $P=MN$ as above and let ${\bf b} = {\bf a}_{I_P}$ be an alcove as above, of the form ${\bf a}_{I_P} = \nu + {\bf a}$, where $\nu$ is $M$-central and sufficiently $N$-dominant. The following theorem is the first main result of this article. It makes explicit the results in \cite[Cor.\,10.2]{Hai25}, and follows directly from equation (\ref{refined_union_I}) and Lemma \ref{base_case_decomp}.

\begin{thm} \label{Thm_A_body}
The \textup{(}reduced\textup{)} parabolic Mirkovi\'c-Vilonen intersection can be explicitly paved by locally closed $k$-schemes of the following form
\begin{equation} \label{MV_AW_eq}
K_P t^\lambda K/K \, \cap \, Kt^\mu K/K \cong \bigsqcup_{{\bf a}_\bullet \in \mathcal P^{{\bf a}_{I_P}}_{\mu}(\lambda)} \mathbb A_k^{c^+({\bf a}_\bullet)} \times (\mathbb A^1_k - \mathbb A^0_k)^{f^+({\bf a}_\bullet)}.
\end{equation}
\end{thm} 
This theorem gives us an alcove-walk description of the parabolic Mirkovi\'c-Vilonen intersections. Using Proposition \ref{K_P_upper_bd_prop} we deduce:

\begin{cor} \label{dim_bound_P}
For any ${\bf a}_\bullet \in \mathcal P^{{\bf a}_{I_P}}_\mu(\lambda)$, we have ${\rm dim}({\bf a}_\bullet) \leq \langle \rho, \mu + \lambda \rangle$. 
\end{cor}

It is now natural to define the set of {\em maximal dimension} alcove walks in $\mathcal P^{{\bf a}_{I_P}}_\mu(\lambda)$ to be
\begin{equation} \label{max_a_I_P_def}
\mathcal M^{{\bf a}_{I_P}}_\mu(\lambda) = \Big\{ {\bf a}_\bullet \in \mathcal P^{{\bf a}_{I_P}}_\mu(\lambda)  ~ | ~ {\rm dim}({\bf a}_\bullet) = \langle \rho, \mu + \lambda \rangle \Big\}.
\end{equation}
The next result is a characterization of the set $\mathcal M^{{\bf a}_{I_P}}_\mu(\lambda)$ inside the set  $\mathcal P^{{\bf a}_{I_P}}_\mu(\lambda)$.

\begin{cor} If ${\bf a}_\bullet \in \mathcal P^{{\bf a}_{I_P}}_{(t_{-w(\mu)})_{\bf 0}}(-w''(\lambda))$, then 
$$
c^-({\bf a}_\bullet) \geq \ell((t_{-w(\mu)})_{\bf 0}) - \langle \rho, \mu + \lambda \rangle,
$$
and 
${\bf a}_\bullet \in  \mathcal M^{{\bf a}_{I_P}}_\mu(\lambda)$ if and only if
$$
c^-({\bf a}_\bullet) =  \ell((t_{-w(\mu)})_{\bf 0}) -  \langle \rho, \mu + \lambda \rangle.
$$
\end{cor}
\begin{rmk} In some cases the quantity $ \ell((t_{-w(\mu)})_{\bf 0}) -  \langle \rho, \mu + \lambda \rangle$ is negative (see Proposition \ref{appendix_prop}). Of course if that happens it cannot coincide with $c^-({\bf a}_\bullet) \geq 0$, and such $w$ will not contribute to the set  $\mathcal M^{{\bf a}_{I_P}}_\mu(\lambda)$.
\end{rmk}

\begin{rmk}
We can write $\ell((t_{-w(\mu)})_{\bf 0}) = \langle \rho, 2\mu \rangle - \epsilon_{\mu,w}$ for a unique integer $\epsilon_{\mu,w}$ with $0 \leq \epsilon_{\mu,w} \leq \ell(w_0)$, where $w_0 \in W_0$ is the longest element (see Lemma \ref{type_elements}).  Then we may rewrite the above: ${\bf a}_\bullet \in  \mathcal P^{{\bf a}_{I_P}}_{(t_{-w(\mu)})_{\bf 0}}(-w''(\lambda))$ always satisfies 
$$
c^-({\bf a}_\bullet) \geq \langle \rho, \mu - \lambda \rangle - \epsilon_{\mu,w},
$$
 and ${\bf a}_\bullet \in \mathcal M^{{\bf a}_{I_P}}_\mu(\lambda)$ if and only if equality holds.
\end{rmk}

When ${\bf a}_\bullet \in \mathcal M^{{\bf a}_{I_P}}_\mu(\lambda)$ then it corresponds to an irreducible component $C_{{\bf a}_\bullet} $ in the parabolic Mirkovi\'c-Vilonen variety $K_Pt^\lambda K/K \, \cap \, Kt^\mu K/K$, and $C_{{\bf a}_\bullet}$ contains a variety isomorphic to the product $\mathbb A_k^{c^+({\bf a}_\bullet)} \times (\mathbb A^1_k - \mathbb A^0_k)^{f^+({\bf a}_\bullet)}$ as an open dense subvariety. The following is the second main result of this article, and it rephrases the above discussion.

\begin{thm} \label{main_thm_P}
 Let $\mu \in X_*(T)^+$ and $\lambda \in X_*(T)^{+_M}$ and assume $\mu - \lambda$ is in the coroot lattice for $G$. Then: 
\begin{enumerate}
\item[(1)] There is a bijection ${\bf a}_\bullet \mapsto C_{{\bf a}_\bullet}$ 
\begin{equation} \label{bij_P_1}
\mathcal M^{{\bf a}_{I_P}}_\mu(\lambda) ~~ \overset{\sim}{\longrightarrow} ~~ {\rm Irred}^{\langle \rho, \mu + \lambda \rangle}(K_Pt^\lambda K/K \, \cap \, Kt^\mu K/K).
\end{equation} 
\item[(2)] The multiplicity $[V^{\widehat{G}}_\mu \, : \, V^{\widehat{M}}_\lambda]$ is the number of ${\bf a}_{I_P}$-positively folded alcove walks  ${\bf a}_\bullet$ of type $(t_{-w\mu})_{\bf 0}$ for some $w \in W_0/W_{0,-\mu}$,  joining ${\bf a}$ to $-\lambda_w + {\bf 0}$, and with dimension $c^+({\bf a}_\bullet) + f^+({\bf a}_\bullet)$ equal to the maximum possible value $\langle \rho, \mu + \lambda \rangle$. 
\end{enumerate}
\end{thm}

\begin{rmk}
There are nonempty intersections $K_P t^\lambda K/K \cap K t^\mu K/K$ which have dimension strictly less than the upper bound $\langle \rho, \mu + \lambda \rangle$ (see Example \ref{branching_counterexample} below).   In favorable cases, however, the nonempty intersections are equidimensional of dimension $\langle \rho, \mu + \lambda \rangle$. This holds when $P = B$ (see Theorem \ref{main_thm_B}); it also holds for all parabolic subgroups $P$ of $G = {\rm GL}_n$, and of more general groups $G$ whenever $\mu$ happens to be a sum of minuscule cocharacters (this follows from \cite[Thm.\,12.7]{HKM} and \cite[Lem.\,9.3]{HKM}).
\end{rmk}

\begin{ex}\label{branching_counterexample}  In order to find a group $G$, a Levi subgroup $M$, and $\mu, \lambda$ such that the intersection $K_P t^\lambda K/K \, \cap \, K t^\mu K/K$ has dimension smaller than $\langle \rho, \mu + \lambda \rangle$, it is enough to find $\widehat{G} \supset \widehat{M} \supset \widehat{T}$ and $\widehat{T}$-weights $\mu, \lambda$ such that $\lambda$ appears in $V^{\widehat{G}}_{\mu} |_{\widehat{T}}$ but $V^{\widehat{M}}_{\lambda}$ does not appear in $V^{\widehat{G}}_\mu|_{\widehat{M}}$.  I thank Jeffrey Adams for finding the following example, which is perhaps the simplest possible. 

Let $\widehat{G} = {\rm Sp}(4)$, and denote the simple roots following the Bourbaki numbering as $\alpha_1 = \epsilon_1 - \epsilon_2$ and $\alpha_2 = 2\epsilon_2$.  Let $\mu = \bar{\omega}_2 = \epsilon_1 + \epsilon_2$, the second fundamental weight, which is the unique quasi-minuscule weight. Let $\lambda = 0$. The quasi-minuscule representation $V^{\widehat{G}}_\mu$ of ${\rm Sp}(4)$ has dimension 5, and the weights are the four elements in the Weyl group orbit of $\mu$, along with the origin $\lambda = 0$, with multiplicity 1. Let $\widehat{M}$ be the Levi subgroup whose simple root is $\alpha_1$. Then
$$
V^{\widehat{G}}_{\mu}|_{\widehat{M}} = W_1 \oplus W_2 \oplus W_3,
$$
where $W_1$ is the 1-dimensional representation of $\widehat{M}$ with weight $\mu$, $W_2$ is the 1-dimensional representation of $\widehat{M}$ with weight $-\mu$, and $W_3$ is the irreducible 3-dimensional representation of $\widehat{M}$ with weights $s_{\alpha_2}(\mu) = \alpha_1$, $0$, and $s_{\alpha_1}s_{\alpha_2}(\mu) = -\alpha_1$.  Note that $\lambda = 0$ appears as an $\widehat{M}$-dominant weight, but the trivial representation $V^{\widehat{M}}_\lambda$ does not appear in $V^{\widehat{G}}_{\mu}|_{\widehat{M}}$.

Theorem \ref{main_thm_P} also recovers this result: with the choices above, the set $\mathcal M^{{\bf a}_{I_P}}_\mu(\lambda)$ is easily seen to be empty.  The above analysis does not prove that the intersection is actually nonempty.  But Theorem \ref{Thm_A_body} provides this, as it is also easily seen in this case that $\mathcal P^{{\bf a}_{I_P}}_\mu(\lambda)$ is nonempty.
\end{ex}

\subsection{Example: the Mirkovi\'c-Vilonen intersections and weight multiplicities}

We explicate the above theorem in the base $P= B$ and thus $M=T$ and $K_P = LU K_T$, $I_P = LU I_T$, and ${\bf a}_{I_P}$ is any alcove ${\bf a}_U$ which is sufficiently deep inside the dominant Weyl chamber $\mathcal C$. We consider any $\mu \in X_*(T)^+$ and $\lambda \in X_*(T)$. Moreover 
$$
[V^{\widehat{G}}_\mu \, : \, V^{\widehat{T}}_\lambda] = {\rm dim}\,V_\mu(\lambda)
$$
where $V_\mu(\lambda)$ is the $\lambda$-weight space in $V_\mu$. Finally, it is well-known that whenever it is nonempty (i.e., whenever $\lambda$ appears as a $\widehat{T}$-weight in $V^{\widehat{G}}_\mu$), the Mirkovi\'c-Vilonen intersection $LU t^\lambda K/K \, \cap \, Kt^{\mu} K/K$ is equidimensional of dimension $\langle \rho, \mu + \lambda \rangle$ (\cite{NP01,MV07}; see also \cite{BaRi18}).  We deduce:

\begin{thm} \label{main_thm_B} Let $\mu \in X_*(T)^+$ and $\lambda \in X_*(T)$.  Assume $\mu - \lambda$ is in the coroot lattice for $G$. 
\begin{enumerate}
\item[(1)] There is a bijection ${\bf a}_\bullet \mapsto C_{{\bf a}_\bullet}$ 
\begin{equation} \label{bij_B_1}
\mathcal M^{{\bf a}_U}_\mu(\lambda) ~~ \overset{\sim}{\longrightarrow} ~~ {\rm Irred}(LUt^\lambda K/K \, \cap \, Kt^\mu K/K). \notag
\end{equation} 
\item[(2)] The weight multiplicity ${\rm dim}\,V_\mu(\lambda)$ is the number of ${\bf a}_U$-positively folded alcove walks  ${\bf a}_\bullet$ of type $(t_{-w\mu})_{\bf 0}$ for some $w \in W_0/W_{0,-\mu}$,  joining ${\bf a}$ to $-\lambda + {\bf 0}$, and having dimension $\langle \rho, \mu + \lambda \rangle$, or equivalently, having the number $c^-({\bf a}_\bullet)$ of ${\bf a}_U$-negative crossings the minimum possible, namely $c^-({\bf a}_\bullet) = \ell((t_{-w(\mu)})_{\bf 0}) -  \langle \rho, \mu + \lambda \rangle$.
\end{enumerate}
\end{thm}

\begin{rmk}
For given $\lambda$, there is only a limited set of types $(t_{-w(\mu)})_{\bf 0}$ we have to consider.  In fact it is easy to see that $w(\mu)$ contributes to $\calP^{{\bf a}_U}_\mu(\lambda)$ only if $-\lambda \preceq -w(\mu)$, where $\lambda_1 \preceq \lambda_2$ is defined to mean that $\lambda_2 - \lambda_1$ is a $\mathbb Z_{\geq 0}$-linear combination of simple coroots. This is illustrated by the example in $\S$\ref{A2_example}.
\end{rmk}

\subsection{The PRV conjecture for branching to Levi subgroups}

The PRV Conjecture for tensor product multiplicities was first proved independently by Kumar \cite{Kum90} and Mathieu \cite{Mat89}. We will use the alcove walk model to prove the (much easier) analogue of the PRV conjecture, for branching to Levi subgroups.  The author also found a more elementary proof using Stembridge's lemma, but the alcove walk proof is given because of useful side consequences.

\begin{prop}  \label{PRV_prop}
Suppose $\mu \in X_*(T)$ and $\lambda \in W\mu \cap X_*(T)^{+_M}$.  Then $[V^{\widehat{G}}_\mu : V^{\widehat{M}}_\lambda] = 1$.
\end{prop}

\begin{proof}
We shall use Theorem \ref{main_thm_P}(2). For each $w_1 \in W^\mu_0$ choose once and for all a reduced expression $\dot{(t_{-w_1(\mu)})_{\bf 0}}$. We first consider the case $M = T$.  Write $\lambda = w(\mu)$ for $w \in W_0^\mu$.  In this case we wish to prove that ${\rm dim}\, V_\mu(\lambda) = 1$.  Of course this is well-known, but we wish to see it as a consequence of the alcove walk model in order to extract some useful information from the proof.  Since the multiplicities are invariant under $W_0$-conjugation, we only need to consider the case $\lambda = \mu$. In this case $t_{-\mu} = (t_{-\mu})_{\bf 0}$,  and there is a unique ${\bf a}_U$-positively folded alcove walk ${\bf a}_\bullet$ from ${\bf a}$ to $-\mu + {\bf a}$ of type $\dot{t}_{-\mu}$ (the unique one given by the reduced word expression itself) and since it consists only of positive crossings, its dimension is $\ell(t_{-\mu}) = \langle \rho, \mu + \mu \rangle$; further no other $w(\mu) \neq \mu$ can contribute an alcove walk terminating at $-\mu + {\bf a}$ because $\ell((t_{-w(\mu)})_{\bf 0}) < \ell(t_{-\mu})$ by Lemma \ref{type_elements}. 

Now we extract a consequence for any $\lambda \in W\mu$. Write $\lambda = w(\mu)$ for $w \in W^\mu_0$. By the invariance mentioned above we know there is a unique ${\bf a}_U$-positively folded alcove walk ${\bf a}_\bullet$ from ${\bf a}$ to $-w(\mu) + {\bf 0}$ which has dimension $\langle \rho, \mu + w(\mu)\rangle$ and type of the form $\dot{(t_{-w_1(\mu)})_{\bf 0}}$, for some $w_1 \in W_0$. We claim, and will prove by induction on $\ell(w)$, that ${\bf a}_\bullet$ is the one coming from $\dot{(t_{-w(\mu)})_{\bf 0}}$ itself.  We proved this for $w = 1$ above, so assume $w(\mu) \neq \mu$. If ${\bf a}_\bullet$ is not as claimed, then it comes from some $\dot{(t_{-w_1(\mu)})_{\bf 0}}$ with $w_1 \in W^\mu_0$ ($w_1 \neq w$); note that $(t_{-w(\mu)})_{\bf 0} \leq  (t_{-w_1(\mu)})_{\bf 0}$ and hence  $w_1 < w$ by Lemma \ref{type_elements} and \cite[Lem.\,7.5]{HN02}.  Then ${\bf a}_\bullet$ results from the alcove walk from  $\dot{(t_{-w_1(\mu)})_{\bf 0}}$ by successively replacing (in order of the alcove walk) certain negative crossings with positive foldings. Each such replacement {\em can only happen at a negative crossing for a hyperplane give by a finite root}: otherwise, the vertex $-w_1(\mu)$ would get folded to an end-vertex which is strictly inside the convex hull of $W_0(-\mu)$, and therefore not equal to $-w(\mu)$.  On the other hand, we may choose the reduced word $\dot{(t_{-w_1(\mu)})_{\bf 0}}$ to be a product of form $\dot{v}\dot{u}\tau$, where $\tau \in \Omega$ is such that $t_{-\mu} \in W_{\rm aff} \tau$, as follows. Let $w_* \in W_0$ be the unique element such that $-\mu +{\bf a} \subset w_*(\mathcal C)$.  Let $\dot{v}$ be a reduced word for $v := w_1w_*$, and $\dot{u}$ a reduced word for  $u := v^{-1}(t_{-w_1(\mu)})_{\bf 0} = w^{-1}_* t_{-\mu}$ (use Lemma \ref{type_elements}). Note that the alcove walk from ${\bf a}$ to $w_*^{-1}t_{-\mu}{\bf a}$ corresponding to $\dot{u}$ lies entirely inside $\mathcal C$. It follows that the alcove walk from $w_1w_*{\bf a}$ to $w_1w_* u {\bf a} = (t_{-w_1(\mu)})_{\bf 0}{\bf a}$ corresponding to $\dot{u}$ lies entirely inside the Weyl chamber $w_1w_*(\mathcal C)$.  Therefore the alcove walk for $\dot{v}\dot{u}\tau$ from ${\bf a}$ to $(t_{-w_1(\mu)})_{\bf 0}{\bf a}$ has no negative crossings through any finite root hyperplane. If any of its actual negative crossings are replaced by foldings, the end-vertex could again not be $-w(\mu)$, by the convex hull reasoning above. We get a contradiction, hence the claim is proved. Thus the only possible alcove walk of the correct dimension is the one from $\dot{(t_{-w(\mu)})_{\bf 0}}$. In fact this also proves the corollary below.

Now we use what we have learned to handle the general case. We assume $\lambda \in W\mu \cap X_*(T)^{+_M}$. We write $\lambda = w(\mu)$ for some $w \in W^\mu_0$. We proved above the (known) result that ${\rm dim}\,V^{\widehat{G}}_\mu(\lambda) = 1$.  So $[V^{\widehat{G}}_\mu : V^{\widehat{M}}_\lambda] \leq 1$, and it suffices to find an ${\bf a}_{I_P}$-positively folded alcove from ${\bf a}$ to $-\lambda + {\bf 0}$ which has dimension $\langle \rho, \mu + \lambda \rangle$.  The alcove walk coming from $\dot{(t_{-\lambda})_{\bf 0}}$ has the required dimension as we saw above. By construction it is ${\bf a}_U$-positively folded and also ${\bf a}_{I_P}$-positively folded. This completes the proof.
\end{proof}

\begin{cor} [of proof] \label{separation_cor}
For each $w$, the alcove walk from ${\bf a}$ to $(t_{-w(\mu)})_{\bf 0}{\bf a}$ based on any reduced word expression $\dot{(t_{-w(\mu)})_{\bf 0}}$ has dimension $\langle \rho, \mu + w(\mu) \rangle$, and moreover this dimension is the number $N_{\mu, w}$ of affine hyperplanes separating the alcove $(t_{-w(\mu)})_{\bf 0} {\bf a}$ from the Weyl chamber $\mathcal C$.
\end{cor}

Our convention is that a hyperplane that {\em separates} these two subsets is allowed to intersect the closure of one or both of them.

\begin{proof}
Let ${\bf a}_\bullet$ be the alcove walk coming from $\dot{(t_{-w(\mu)})_{\bf 0}}$; since the word is reduced we see that $f^+({\bf a}_\bullet) = 0$. We have 
$$
c^+({\bf a}_{\bullet}) = {\rm dim}\,{\bf a}_{\bullet} = \langle \rho, \mu + w(\mu) \rangle.
$$
In this case the dimension is the number of hyperplanes separating $(t_{-w(\mu)})_{\bf 0} {\bf a}$ from ${\bf a}$ and also from ${\bf a}_{U}$, or equivalently, separating it from the entire Weyl chamber $\mathcal C$.
\end{proof}

\subsection{Example: the case $M = G$}

In this case $K_P = K$ and by the Cartan decomposition the variety $K_P t^\lambda K/K \, \cap \, Kt^\mu K/K/$ is nonempty (and then irreducible) if and only if $\lambda = \mu$.  To see this from the point of view of alcove walks, note that since ${\bf a}_{I_P} = {\bf a}$, the only ${\bf a}$-positively folded alcove walk of type $(t_{-w(\mu)})_{\bf 0}$ is the walk underlying the word we start with.  Assuming the intersection is nonempty, we have some $w$ and some $w''$ with $w''(\lambda) = w(\mu)$, from which it follows that $\lambda = \mu$ and $w=w'' \in W_0/W_{0,-\mu}$. Now $\ell((t_{-w(\mu)})_{\bf 0})- \langle \rho, \mu + \lambda \rangle$ is negative for all $w$ with $w(\mu) \neq \mu$ and this difference is $0$ when $w(\mu) = \mu$.  There is thus exactly one alcove walk ${\bf a}_\bullet$ which contributes, and it has $c^-({\bf a}_\bullet) = 0$.

\subsection{Type $A_2$ example} \label{A2_example}

We consider $G = {\rm GL}_3$. Let $P=B =TU \supset T$ be the standard Borel pair, where $T$ is the diagonal torus and $B$ (resp.\,$U$) the upper (resp.\,strictly upper) triangular invertible matrices. Let $I$ denote the usual ``upper triangular" Iwahori subgroup. Consider the dominant cocharacter $\mu = (3,1,0) \in \mathbb Z^3 \cong X_*(T)$.   We consider a $W_0$-orbit $\{\lambda_1, \lambda_2, \lambda_3\}$ of translations in the $T$-weight space for $\mu$, where we set
\begin{equation}
\lambda_1 = (1,1,2) \hspace{.3in} \lambda_2 = (2,1,1) \hspace{.3in} \lambda_3 = (1,2,1).
\end{equation}

The base alcove ${\bf a}$ gives rise to the simple affine reflections $s_1 = s_{\alpha_1}$, $s_2 = s_{\alpha_2}$, and $s_0 = t_{\tilde{\alpha}^\vee} s_{\tilde{\alpha}}$, giving the Coxeter system structure to the affine Weyl group $W_{\rm aff} = \langle s_0, s_1, s_2\rangle$.  Here $\alpha_1 = \epsilon_1 - \epsilon_2$, $\alpha_2 = \epsilon_2 - \epsilon_3$, and $\tilde{\alpha} = \epsilon_1 - \epsilon_3$. Recall the decomposition of the extended affine Weyl group $W = W_{\rm aff} \rtimes \Omega$; let $\tau \in \Omega$ be the unique element such that $t_{-\mu} \in W_{\rm aff} \rtimes \tau$. Let $w_0 = s_1s_2s_1 = s_2s_1s_2$ denote the longest element of the finite Weyl group $W_0 = \langle s_1, s_2 \rangle \cong S_3$.

\begin{figure}
\begin{center}
 \resizebox{4.in}{!}
{
\begin{overpic}{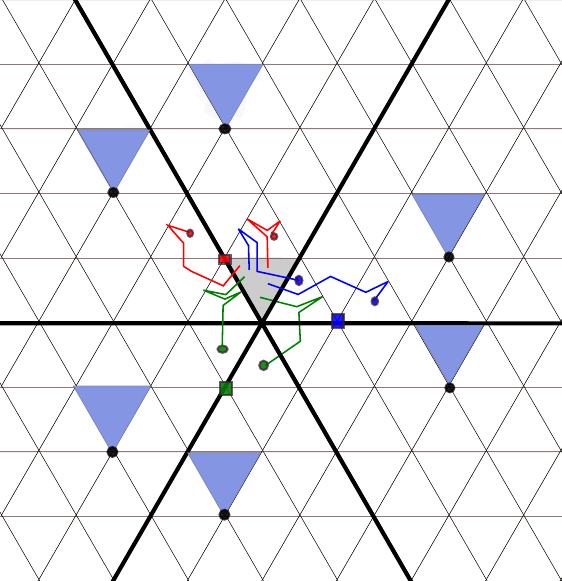}
\put(23.25,62.5){\tiny \cm $1012$}
\put(31.25,50.5){\tiny \cp $10$}
\put(56.25,49.5){\tiny \cp $20$}
\put(68.25,50.5){\tiny \cm $2012$}
\put(56.25,37.5){\tiny \cp $210$}
\put(67.65,36.5){\tiny \cm $21012$}
\put(62,40){\tiny \cp $2101$}
\put(62.85,45.5){\tiny \cp $201$}
\put(30.6,37.5){\tiny \cp $120$}
\put(39.25,40.5){\tiny \cp $12$}
\put(48.75,40.5){\tiny \cp $21$}
\put(54,46){\tiny \cp $2$}
\put(29.75,29.5){\tiny \cp $1201$}
\put(43.25,34.15){\tiny \cp $212$}
\put(42.75,29.5){\tiny \cp $1210$}
\put(22.75,25){\tiny \cm $12012$}
\put(35.5,24){\tiny \cp $12101$}
\put(34.70,18.55){\tiny \cm $121012$}
\put(34.75,7.5){${\bf -\mu}$}
\put(32.75,56.5){\tiny \cp $101$}
\put(35.25,46){\tiny \cp $1$}
\put(37,70){\tiny \cm $012$}
\put(37.75,62.5){\tiny \cp $01$}
\put(50.5,62.5){\tiny \cp $02$}
\put(44.75,62.5){\tiny \cp $0$}
\end{overpic}
}
\caption{The base alcove ${\bf a}$ is colored in light gray, and the bolded vertices are the elements in $W_0(-\mu)$, with corresponding translation alcoves $-w(\mu) + {\bf a}$ colored in light blue. Each alcove $(t_{-w(\mu)})_{\bf 0}({\bf a})$ is labelled in magenta with the $W_{\rm aff}$-part of a chosen reduced word expression (e.g.,\,we write $12012$ in place of $s_{12012}\tau$).  Other relevant alcoves have their reduced words colored in light pink. The three vertices $-\lambda_i + {\bf 0}$ are colored: $-\lambda_1 = \mbox{\cred red square}$, $-\lambda_2 = \mbox{\cg green square}$, and $-\lambda_3 = \mbox{\cb blue square}$. The alcove walks of maximum possible dimension which terminate at $-\lambda_i + {\bf 0}$ are given the same color as that vertex.}{\label{fig:fig1}}
\end{center}
\end{figure}
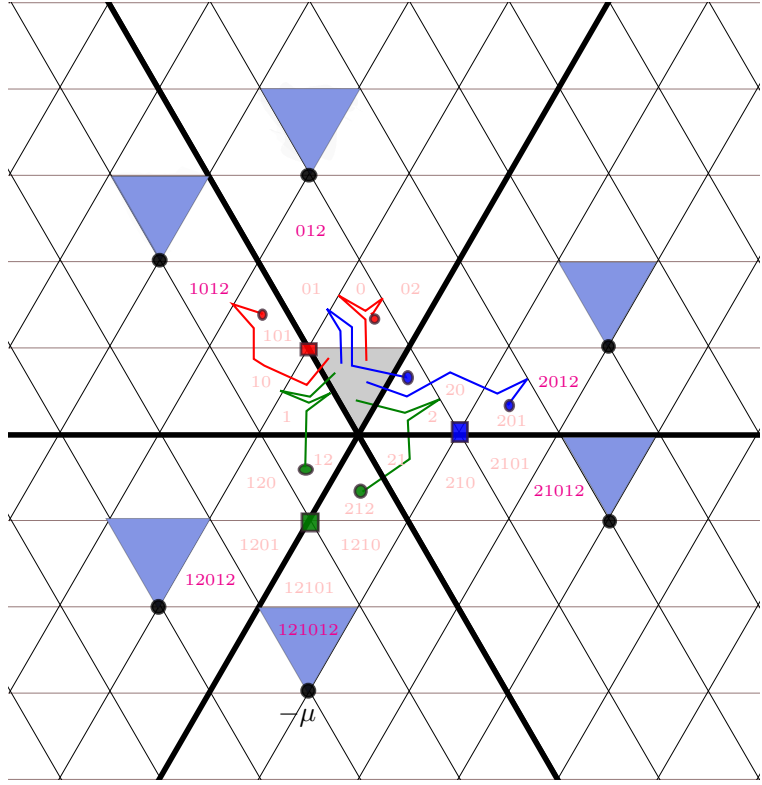

There are six elements in the Weyl-orbit $W_0(-\mu)$, and below we choose one reduced  expression for each of the right $W_0$-minimal elements $(t_{-w(\mu)})_{\bf 0}$.  We abbreviate by writing $s_0 s_1 s_2 \tau$ simply as $s_{012}\tau$, etc.\,:
\begin{align*}
t_{-\mu} =(t_{-\mu})_{\bf 0} &= s_{121012} \tau\\
(t_{-s_2(\mu)})_{\bf 0} &= s_{12012} \tau \\
(t_{-s_1(\mu)})_{\bf 0} &= s_{21012} \tau \\
(t_{-s_{12}(\mu)})_{\bf 0} &= s_{2012}\tau \\
(t_{-s_{21}(\mu)})_{\bf 0} &= s_{1012}\tau = s_{0102} \tau \\
(t_{-w_0(\mu)})_{\bf 0} &= s_{012} \tau.
\end{align*}
Let ${\bf a}_U$ be an alcove sufficiently deep inside the dominant Weyl chamber.  Fix $\lambda_i$. We may list all the {\em maximum-possible-dimensional} ${\bf a}_U$-positively folded alcove walks starting from ${\bf a}$ and terminating at the vertex $-\lambda_i + {\bf 0}$, having type given by one of the six reduced words listed above.  It turns out that for each $\lambda_i$, only two types give rise to a maximum-possible-dimensional alcove walk terminating at $-\lambda_i + {\bf 0}$. We list all such alcove walks which arise, as well as the open cellular subvariety in the corresponding irreducible component of $LU t^{\lambda_i} K/K \cap K t^\mu K/K$:

\noindent $-\lambda_1$:

\noindent Type $s_{012}\tau$: \,\,$\tau \overset{c^-_{s_0}}{\longrightarrow} s_0\tau \overset{f^+_{s_1}}{\longrightarrow} s_0\tau \overset{f^+_{s_2}}{\longrightarrow} s_0\tau \hspace{.85in} \rightsquigarrow \hspace{.85in} (\mathbb A^1_k)^0 \times (\mathbb A^1_k - \mathbb A^0_k)^2$.

\noindent Type $s_{1012}\tau$: \,\,$\tau \overset{c^+_{s_1}}{\longrightarrow} s_1\tau \overset{c^-_{s_0}}{\longrightarrow} s_{10}\tau \overset{c^-_{s_1}}{\longrightarrow} s_{101}\tau \overset{f^+_{s_2}}{\longrightarrow} s_{101}\tau \hspace{.43in} \rightsquigarrow \hspace{.43in}  (\mathbb A^1_k)^1 \times (\mathbb A^1_k - \mathbb A^0_k)^1$.

\bigskip

\noindent $-\lambda_2$:

\noindent Type $s_{2012}\tau$: \,\,$\tau \overset{c^+_{s_2}}{\longrightarrow} s_2\tau \overset{f^+_{s_0}}{\longrightarrow} s_2\tau \overset{c^+_{s_1}}{\longrightarrow} s_{21}\tau \overset{c^+_{s_2}}{\rightarrow} s_{212}\tau \hspace{.5in} \rightsquigarrow \hspace{.5in} (\mathbb A^1_k)^3 \times (\mathbb A^1_k - \mathbb A^0_k)^1$.

\noindent Type $s_{1012}\tau$: \,\,$\tau \overset{c^+_{s_1}}{\longrightarrow} s_1\tau \overset{f^+_{s_0}}{\longrightarrow} s_{1}\tau \overset{f^+_{s_1}}{\longrightarrow} s_{1}\tau \overset{c^+_{s_2}}{\longrightarrow} s_{12}\tau \hspace{.52in} \rightsquigarrow \hspace{.52in}  (\mathbb A^1_k)^2 \times (\mathbb A^1_k - \mathbb A^0_k)^2$.

\bigskip

\noindent $-\lambda_3$:

\noindent Type $s_{2012}\tau$: \,\,$\tau \overset{c^+_{s_2}}{\longrightarrow} s_2\tau \overset{c^-_{s_0}}{\longrightarrow} s_{20}\tau \overset{c^+_{s_1}}{\longrightarrow} s_{201}\tau \overset{f^+_{s_2}}{\rightarrow} s_{201}\tau \hspace{.45in} \rightsquigarrow \hspace{.45in} (\mathbb A^1_k)^2 \times (\mathbb A^1_k - \mathbb A^0_k)^1$.

\noindent Type $s_{0102}\tau$: \,\,$\tau \overset{c^-_{s_0}}{\longrightarrow} s_0\tau \overset{f^+_{s_1}}{\longrightarrow} s_{0}\tau \overset{c^+_{s_0}}{\longrightarrow} \tau \overset{c^+_{s_2}}{\longrightarrow} s_{2}\tau \hspace{.62in} \rightsquigarrow \hspace{.62in}  (\mathbb A^1_k)^2 \times (\mathbb A^1_k - \mathbb A^0_k)^1$.

\medskip

Note that in the last line, we used the reduced word $s_{0102}\tau$ instead of the reduced word $s_{1012}\tau$ appearing in the previous cases.  This is to make the picture in Figure \ref{fig:fig1} more legible. There are a few alcove walks in $\calP^{{\bf a}_U}_\mu(\lambda)$ which do not have maximum possible dimension.  We do not list these, but it is easy to do so.

The above is consistent with the known multiplicities ${\rm dim}\,V^{{\rm GL}_3}_\mu(\lambda_i) = 2$ for $i =1,2,3$.

\medskip

\section{Variant: Alcove walk model for tensor product multiplicities} \label{variant_sec}

Fix three elements $\lambda, \mu, \nu \in X_*(T)^+$ such that $\mu + \lambda - \nu$ belongs to the coroot lattice. Let $\mu^* = -w_0(\mu)$. Then the intersection
$$
{\rm Conv}(\lambda, \mu; \nu) := K t^\lambda K/K \, \cap \, t^\nu K t^{\mu^*} K/K
$$
can be identified with the fiber of a convolution morphism, as in \cite[Lem.\,7.3]{Hai25}. We shall give a description of this space in terms of alcove walks. It is well-known that $${\rm dim}\,{\rm Conv}(\lambda, \mu; \nu) \leq \langle \rho, \mu + \lambda - \nu \rangle$$ and that
$$
\#{\rm Irred}^{ \langle \rho, \mu + \lambda - \nu \rangle}({\rm Conv}(\lambda, \mu; \nu)) = [V^{\widehat{G}}_\mu \otimes V^{\widehat{G}}_\lambda \, : \, V^{\widehat{G}}_\nu],
$$
the multiplicity of $V^{\widehat{G}}_\nu$ in the representation $V^{\widehat{G}}_\mu \otimes V^{\widehat{G}}_\lambda$. This is a well-known consequence of the geometric Satake equivalence \cite{MV07}, see e.g.\,\cite{Hai03, BaRi18}. 

Using \cite[Lem.\,7.2]{Hai25} as in the proof of Lemma \ref{key_P-MV_lem}, we get a paving
\begin{equation} \label{Conv_I_decomp_eq}
{\rm Conv}(\lambda, \mu,; \nu) \cong \bigsqcup_{w \in W_0/W_{0,-\mu^*}} \bigsqcup_{w' \in W_0} \bigsqcup_{w'' \in W_0/W_{0, -\lambda}}  I w'' t_{-\lambda}w' I/I \, \cap \, t_{-\nu} I (t_{-w(\mu^*)})_{\bf 0} I/I.
\end{equation}

Now the proof of the main theorem in \cite{Hai25} yields the following alcove-walk description of the convolution fiber.

\begin{thm} \label{Conv_AW_model}
The reduced convolution fiber ${\rm Conv}(\lambda, \mu; \nu)$ has a paving by locally closed subvarieties of the form
$$
{\rm Conv}(\lambda, \mu; \nu) \cong ~~\bigsqcup_{{\bf a}_\bullet} ~~\mathbb A^{c^+({\bf a}_\bullet)}_k \times (\mathbb A^1_k - \mathbb A^0_k)^{f^+({\bf a}_\bullet)},
$$
where each ${\bf a}_\bullet$ belongs, for some $w'', w$ as above, to  the set $\calP^{\bf a}_{(t_{-w(\mu^*)})_0}(-\nu + {\bf a}, -w''(\lambda) + {\bf 0} )$ of ${\bf a}$-positively folded labeled alcove walks of type $(t_{-w(\mu^*)})_0$ from the alcove $t_{-\nu} {\bf a}$ to one of the vertices $-w''(\lambda) + {\bf 0}$.  
\end{thm}
Clearly it follows from this that for every ${\bf a}_\bullet$ appearing which is attached to $w'', w$, we have 
\begin{align*}
c^+({\bf a}_\bullet) + f^+({\bf a}_\bullet) &= {\rm dim}({\bf a}_\bullet) \leq \langle \rho, \mu + \lambda - \nu \rangle, ~~\mbox{or equivalently,} \\
c^-({\bf a}_\bullet) &\geq \ell((t_{-w(\mu^*)})_0) - \langle \rho, \mu + \lambda - \nu \rangle.
\end{align*}
So it makes sense to define 
$$
\mathcal M^{{\bf a}}_{\mu, \lambda}(\nu) := \coprod_{w'', w} \mathcal M^{{\bf a}}_{(t_{-w(\mu^*)})_0}(t_{-\nu}{\bf a}, -w''(\lambda) + {\bf 0}),
$$
where $\mathcal M^{{\bf a}}_{(t_{-w(\mu^*)})_0}(t_{-\nu}{\bf a}, -w''(\lambda) + {\bf 0})$ consists of those ${\bf a}_\bullet \in \calP^{\bf a}_{(t_{-w(\mu^*)})_0}(-\nu + {\bf a}, -w''(\lambda) + {\bf 0} )$ which satisfy
\begin{align*}
{\rm dim}({\bf a}_\bullet) &= \langle \rho, \mu  + \lambda - \nu \rangle, ~~\mbox{or equivalently,} \\
c^-({\bf a}_\bullet) &= \ell((t_{-w(\mu^*)})_0) - \langle \rho, \mu + \lambda - \nu \rangle.
\end{align*}
As before, if $\ell((t_{-w(\mu^*)})_0) - \langle \rho, \mu + \lambda - \nu \rangle < 0$, then $\calM^{\bf a}_{(t_{-w(\mu^*)})_0}(-\nu + {\bf a}, -w''(\lambda) + {\bf 0} ) = \emptyset$, and $-w(\mu^*)$ does not contribute any alcove walks to the set we are counting.

\begin{thm} \label{Tensor_AW_model} Let $\mu, \lambda, \nu \in X_*(T)^+$ be such that $\mu + \lambda - \nu$ belongs to the coroot lattice.  Then
\begin{enumerate}
\item[(1)] There is a bijection ${\bf a}_\bullet \mapsto C_{{\bf a}_\bullet}$ 
\begin{equation} \label{bij_B_1}
\mathcal M^{{\bf a}}_{\mu, \lambda}(\nu) ~~ \overset{\sim}{\longrightarrow} ~~ {\rm Irred}^{\langle \rho, \mu + \lambda - \nu \rangle}({\rm Conv}(\mu, \lambda; \nu)). \notag
\end{equation} 
\item[(2)] The multiplicity $[V^{\widehat{G}}_\mu \otimes V^{\widehat{G}}_\lambda \, : \, V^{\widehat{G}}_\nu]$ is the number of ${\bf a}$-positively folded alcove walks  ${\bf a}_\bullet$ of type $(t_{-w(\mu^*)})_{\bf 0}$ for some $w \in W_0/W_{0,-\mu^*}$,  joining $t_{-\nu}{\bf a}$ to $-w''(\lambda) + {\bf 0}$ for some $w'' \in W_0$, and such that ${\rm dim}({\bf a}_\bullet) = \langle \rho, \mu + \lambda - \nu \rangle$, or equivalently, the number of ${\bf a}$-negative crossings is the minimum possible, namely $c^-({\bf a}_\bullet) = \ell((t_{-w(\mu^*)})_{\bf 0}) -  \langle \rho, \mu + \lambda - \nu\rangle$.
\end{enumerate}
\end{thm}

\section{Appendix} \label{appendix}

In this appendix, we give additional information about the quantities $\ell((t_{-w(\mu)})_{\bf 0}) - \langle \rho, \mu + \lambda \rangle$ and $\langle \rho, \mu + \lambda \rangle$.   

\begin{lem}  \label{App_Lem1} Let $\lambda_d$ be the unique dominant $W_0$-conjugate of $\lambda$. Write $\lambda = w''(\lambda_d)$ for $w'' \in W_0^{\lambda_d}$.
\begin{enumerate}
\item[(a)] We have the identity
$$
\langle \rho, \lambda_d - w''(\lambda_d) \rangle = \sum_{\underset{w''\alpha < 0}{\alpha > 0}} \langle \alpha, \lambda_d \rangle =: \ell^{\lambda_d}(w'').
$$
We consider $\ell^{\lambda_d}(w'')$ to be a $\lambda_d$-weighted version of the length of $w''$.
\item[(b)] We have $\ell^{\lambda_d}(w'') \geq \ell(w'')$.
\end{enumerate}
\end{lem}

We note that the notion of $\lambda_d$-weighted length was independently discovered by Chapelier-Laget and Gerber (see \cite[Def.\,4.5]{C-LG22}, where it is termed the {\em $\lambda_d$-atomic length}). 

\begin{proof}
Part (a) follows from the relation
$$
\rho - (w'')^{-1}\rho = \sum_{\underset{w''\alpha < 0}{\alpha > 0}} \alpha.
$$
To see part (b) we need to show that $\langle \alpha, \lambda_d \rangle > 0$ whenever $\alpha> 0$ and $w''\alpha < 0$.  If on the contrary $\langle \alpha, \lambda_d \rangle = 0$, then $s_\alpha \in W_{\lambda_d}$ and so $w'' < w'' s_\alpha$, which is equivalent to $w''\alpha > 0$, a contradiction.
\end{proof}

For the next statement, we define $\Omega(\mu) = \{ \nu \in X_*(T) \, | \, w\nu \preceq \mu, \,\,\,\, \forall w \in W_0 \}$, where $\nu_1 \preceq \nu_2$ for $\nu_i \in X_*(T)$ means that $\nu_2 - \nu_1$ is a $\mathbb Z_{\geq 0}$-linear combination of positive coroots. For such a difference, define its {\em height} to be  ${\rm ht}(\nu_2 - \nu_1) := \langle \rho, \nu_2 - \nu_1 \rangle$, an element in $\mathbb Z_{\geq 0}$. Recall from Corollary \ref{separation_cor} that $N_{\lambda_d, w''}$ denotes the number of affine hyperplanes separating $(t_{-w''(\lambda_d)})_{\bf 0}{\bf a}$ from the Weyl chamber $\mathcal C$.

\begin{prop} \label{appendix_prop}
Let $\mu \in X_*(T)^+$ and assume $\lambda \in \Omega(\mu) \cap X_*(T)^{+_M}$. Write $\lambda = w''(\lambda_d)$ for $w'' \in W_0^{\lambda_d}$. Then
$$
\langle \rho, \mu + \lambda \rangle = \ell(t_{\mu}) -  {\rm ht}(\mu - \lambda_d)  - \ell^{\lambda_d}(w'') = {\rm ht}(\mu - \lambda_d) + N_{\lambda_d, w''}.
$$
Furthermore, 
$$
\ell((t_{-w(\mu)})_{\bf 0}) - \langle \rho, \mu + \lambda \rangle = {\rm ht}(\mu - \lambda_d) +  \ell^{\lambda_d}(w'') - \ell(w).
$$
\end{prop}

\begin{proof}
This follows immediately on combining Lemma \ref{type_elements}, Corollary \ref{separation_cor}, and Lemma \ref{App_Lem1}.
\end{proof}

\noindent {\em Statement on Accepted Version:} This version of the article has been accepted for publication, after peer review, but is
not the Version of Record and does not reflect [all] post-acceptance improvements, or [all] corrections. The Version of Record is
available online at: http://dx.doi.org/10.1007/s00031-026-09956-0.


\begin{thebibliography}{999999999}









\bibitem[BaRi18]{BaRi18} P.\,Baumann, S.\,Riche, Simon: {\it Notes on the geometric Satake equivalence}. In: Relative aspects in representation theory, Langlands functoriality and automorphic forms, 1-134. Lecture Notes in Math., 2221. CIRM Jean-Morlet Ser.\,Springer, Cham, 2018.






\bibitem[BD]{BD} A.\, Beilinson, V.\, Drinfeld: {\it Quantization of Hitchin's integrable system and Hecke eigensheaves}, preprint (1999) available at http://www.math.utexas.edu/users/benzvi/Langlands.html. 

\bibitem[BG01]{BG01} A.\,Braverman, D.\,Gaitsgory:{\it Crystals via the affine Grassmannian}, Duke Math.\,J.,\,Vol.\,{\bf 107}, No.\,3, (2001), 561-575.









\bibitem[BV25]{BV25} A.\,Bouthier, E.\,Vasserot: {\it On the geometric Satake equivalence for Kac-Moody groups}, arXiv:2510.11466.






\bibitem[BT72]{BT72} F. Bruhat and J. Tits: {\it Groupes r\'eductifs sur un corps local I. Donn\'ees radicielles valu\'ees}, Inst. Hautes \'Etudes Sci. Publ. Math. 41 (1972), 5-251. 




\bibitem[CP24]{CP24} R.\,Cass, C.\,P\'epin: {\it Constant term functors with ${\mathbb F}_p$-coefficients}, Documenta Mathematica {\bf 29} (2024), 343-397.

\bibitem[CSvdH22]{CSvdH22} R.\,Cass, J.\,Scholbach, T.\,van den Hove: {\it The geometric Satake equivalence for integral motives}, arXiv:2211.04832.




\bibitem[C-LG22]{C-LG22} N.\,Chapelier-Laget, Th.\,Gerber: {\it Atomic Length on Weyl Groups}, arXiv: 2211.12359.
























\bibitem[FS24]{FS24} L.\,Fargues, P.\,Schoze: {\it Geometrization of the local Langlands correspondence}, to appear in Ast\'erisque.  arXiv:2102.13459








\bibitem[GL05]{GL05} S.\,Gaussent, P.\,Littelmann: {\it LS Galleries, The Path Model, and MV Cycles}, Duke Math.\,J.\,vol.\,{\bf 127}, No.\,1, (2005), 35-88.







\bibitem[Goe07]{Goe07} U.\,G\"{o}rtz: {\it Alcove walks and nearby cycles on affine flag manifolds}, J.\,Algebr.\,Comb.\, {\bf 26}, (2007), 415-430.



\bibitem[GHKR06]{GHKR06} U.\,G\"ortz, T.\,Haines, R.\,Kottwitz and D.\,Reuman: {\it Dimensions of some affine Deligne-Lusztig varieties}, Ann. Sci. \'ecole Norm. Sup. (4) \textbf{39} (2006), no. 3, 467-511.

\bibitem[GHKR10]{GHKR10} U.~G\"{o}rtz, T.~Haines, R.~Kottwitz, D.~Reuman: {\em Affine Deligne-Lusztig varieties in affine flag varieties}, Compositio Math. {\bf 146} (2010), 1339-1382.




\bibitem[Hai03]{Hai03} T. Haines: {\it Structure constants for Hecke and representation rings}, Int.\,Math.\,Res.\,Not.\,2003, no. 39, 2103-2119.



\bibitem[Hai09]{Hai09} T.\,Haines: {\it The base change fundamental lemma for central elements in parahoric Hecke algebras}, Duke\,Math\,J., vol.\,{\bf 149}, no.\,3 (2009), 569-643.






\bibitem[Hai25]{Hai25} T.\,Haines: {\em Cellular pavings of fibers of convolution morphisms}, \'Epijournal G\'eom.\,Alg\'ebrique {\bf 9} (2025), Art.\,9, 24 pp.


\bibitem[HKM]{HKM} T.~Haines, M.~Kapovich, J.~J.~Millson, {\it Ideal triangles in Euclidean buildings and branching to Levi subgroups}, J.\,Algebra {\bf 361} (2012), 41-78. 


\bibitem[HKP]{HKP} T.~Haines, R.~Kottwitz, A.~Prasad, \emph{Iwahori-Hecke
algebras}, J.~Ramanujan Math.~Soc.~{\bf 25}, No.2 (2010), 113-145.

\bibitem[HN02]{HN02} T.~Haines, B.~C.~Ng\^{o}, {\em Alcoves associated to special fibers of local models}, Amer. J. Math. {\bf 124} (2002), 1125-1152.


\bibitem[HR08]{HR08} T.\,Haines, M.\,Rapoport: {\it On parahoric subgroups}, Adv.\,Math.\,{\bf 219} (2008), 188-198.







\bibitem[HaRi20]{HaRi20} T.\,Haines, T.\,Richarz:  {\it Smoothness of Schubert varieties in twisted affine Grassmannians}, Duke Math.\,J.\,{\bf 169} (2020), no.\,17, 3223-3260.











\bibitem[KLM08]{KLM08} M. Kapovich, B. Leeb, J. Millson, {\em The generalized triangle inequalities
in symmetric spaces and buildings with applications to algebra},
Memoirs of the AMS {\bf 192},  (2008).







\bibitem[Ka82]{kato} S. Kato, {\em Spherical functions and a $q$-analogue of
Kostant's weight multiplicity formula}, Invent. Math. {\bf 66} (1982), 461-468.










\bibitem[Kum90]{Kum90} S. Kumar: {\it Proof of the Parthasarathy-Ranga Rao-Varadarajan conjecture}, Invent. Math. 102 (1990), no. 2, 377-398.















\bibitem[Li95]{Li95} P.\,Littelmann: {\it Paths and Root Operators in Representation Theory}, Annals of Math.\, 2nd ser.\,vol.\,{\bf 142}, No.\,3 (1995), 499-525.


\bibitem[Lu81]{Lu81} G.\,Lusztig: {\it Singularities, character formulas, and a $q$-analogue of weight multiplicities}, Analysis and Topology on Singular Spaces, II, III, Ast\'erisque, Luminy, 1981, vol. \textbf{101-102}, Soc. Math. France, Paris (1983), pp. 208-229.




\bibitem[Mat89]{Mat89} O.\,Mathieu: {\it Construction d'un groupe de Kac-Moody et applications}, Compositio Math. 69 (1989), 37-60.

\bibitem[MST19]{MST19} E.\,Mili\'cevi\'c, P.\,Schwer, A.\,Thomas: {\it Dimensions of affine Deligne–Lusztig varieties: a
new approach via labeled folded alcove walks and root operators}, pp.\,v+101 in Mem.\,Amer.\,Math.\,Soc. \,1{\bf 260}, Amer.\,Math.\,Soc.,\,Providence, RI, 2019.



\bibitem[MST23]{MST23} E.\,Mili\'cevi\'c, P.\,Schwer, A.\,Thomas: {\it Affine Deligne-Lusztig varieties and folded galleries
governed by chimneys}, Annales de l'Institut Fourier, Tome {\bf 73}, no.\,6, (2023), 2469-2541.

\bibitem[MNST24]{MNST24} E.\,Mili\'cevi\'c, Y.\,Naqvi, P.\,Schwer, A.\,Thomas:{\it A Gallery Model for Affine Flag Varieties via Chimney Retractions}, Transformation Groups {\bf 29}, 773-821 (2024). https://doi.org/10.1007/s00031-022-09726-8

\bibitem[MV07]{MV07} I.\,Mirkovi\'c, K.\,Vilonen: {\it Geometric Langlands duality and representations of algebraic groups over commutative rings},  Ann.\,of Math.\,(2)  \textbf{166}  (2007),  no. 1, 95-143.






\bibitem[NP01]{NP01} Ng\^o B.\,C.\,and P.\,Polo: {\it R\'esolutions de Demazure affines et formule de Casselman-Shalika g\'eom\'etrique}, J.\,Algebraic Geom.\,\textbf{10} (2001), no. 3, 515-547.

\bibitem[Ni22]{Ni22} S.\,Nie: {\it Irreducible components of affine Deligne-Lusztig varieties}, Cambridge Journal of Mathematics {\bf 10}(2), (2022), 433-510.










\bibitem[PRS09]{PRS09} J.\,Parkinson, A.\,Ram, C.\,Schwer: {\it Combinatorics in affine flag varieties}, J.\,Algebra {\bf 321} (2009), 3469-3493.



\bibitem[Ram06]{Ram06} A.\,Ram: {\it Alcove walks, Hecke Algebras, Spherical Functions, Crystals and Column Strict Tableux}, PAMQ, vol.\,{\bf 2}, No.\,4, (2006), 963-1013.








\bibitem[Ri14]{Ri14} T.\,Richarz: {\it A new approach to the geometric Satake equivalence}, Documenta Mathematica \textbf{19} (2014) 209-246.


\bibitem[RS21]{RS21} T.\,Richarz, J.\,Scholbach: {\it The motivic Satake equivalence}, Math.\,Ann.\,{\bf 380} (2021), no.\,3-4, 1595-1653.























\bibitem[vdH24]{vdH24} T.\,van den Hove: {\it The integral motivic Satake equivalence for ramified groups}, arXiv:2404.15694.



\bibitem[Vig16]{Vig16} M.-F.\,Vigneras: {\it The pro-$p$-Iwahori Hecke algebra of a reductive $p$-adic group I},, Compositio Math.\,{\bf 152}, (2016), 693-753.

\bibitem[XZ17+]{XZ17+} L.\,Xiao, X.\,Zhu: {\it Cycles on Shimura varieties via geometric Satake}, arXiv:1707.05700v1. 









\end{thebibliography}
\end{document}